\documentclass[reqno,10pt]{amsart}

\usepackage{amsmath}
\usepackage{amssymb}
\usepackage{amsthm}
\usepackage{mathrsfs}
\usepackage{geometry}
\usepackage{graphicx}
\usepackage{color}
\usepackage{graphics}
\usepackage[latin1]{inputenc}
\usepackage{verbatim}
\usepackage{tikz}
\usepackage{multirow}
\usepackage{expdlist}
\usepackage{subfigure}
\usepackage{hyperref}

\newtheorem{theorem}{Theorem}[section]

\newtheorem{lemma}[theorem]{Lemma}

\newtheorem{corollary}[theorem]{Corollary}

\newtheorem{conjecture}[theorem]{Conjecture}
\newtheorem{proposition}[theorem]{Proposition}

\numberwithin{equation}{section}

\newcommand{\op}{\operatorname{Op}}
\newcommand{\pt}{\operatorname{PT}}
\newcommand{\ps}{\operatorname{PS}}
\newcommand{\qt}{\operatorname{QT}}
\newcommand{\qs}{\operatorname{QS}}

\newcommand{\id}{\operatorname{id}}
\newcommand{\sgn}{\operatorname{sgn}}

\newcommand{\rk}{\operatorname{rk}}

\newcommand{\mn}[1]{C_{n,#1}^{(d)}}
\newcommand{\br}{(k_1,\ldots,k_n)}
\newcommand{\mz}{\mathbb{Z}}
\newcommand{\asm}{\operatorname{ASM}}
\newcommand{\vs}{\operatorname{VSASM}}
\newcommand{\sym}{\operatorname{\mathbf{Sym}}}
\newcommand{\asym}{\operatorname{\mathbf{ASym}}}
\newcommand{\subfl}[2]{\begin{subarray}{l} #1 \\ #2 \end{subarray}}

\begin{document}

\title[Vertically symmetric alternating sign matrices]{Vertically symmetric alternating sign
matrices and a multivariate Laurent polynomial identity}

\author{Ilse Fischer \and Lukas Riegler}
\address{Universit\"at Wien, Fakult\"at f\"ur Mathematik, 
Oskar-Morgenstern-Platz 1, 1090 Wien, Austria}
\email{ilse.fischer@univie.ac.at, lukas.riegler@univie.ac.at}


\thanks{The authors acknowledge the support by the Austrian Science Foundation FWF, START grant Y463}

\begin{abstract} 
In 2007, the first author gave an alternative proof of the refined alternating sign matrix theorem by introducing a linear equation system that determines the refined ASM numbers uniquely. Computer experiments suggest that the numbers appearing in a conjecture concerning the number of  vertically symmetric alternating sign matrices with respect to the position of the first $1$ in the second row of the matrix establish the solution of a linear equation system similar to the one for the ordinary refined ASM numbers. In this paper we show how our attempt to prove this fact naturally leads to a more general conjectural multivariate Laurent polynomial identity. Remarkably, in contrast to the ordinary refined ASM numbers, we need to extend the combinatorial interpretation of the numbers to parameters which are not contained in the combinatorial admissible domain. Some partial results towards proving the conjectured multivariate Laurent polynomial identity and additional motivation why to study it are presented as well.
\end{abstract}

\maketitle

\section{Introduction}

An \emph{Alternating Sign Matrix} (ASM) is  a square matrix with entries in $\{0,1,-1\}$ where in each row and column the
non-zero entries alternate in sign and sum up to $1$. Combinatorialists are especially fond of these objects since they discovered that ASMs belong to the class of objects which possess a simple closed enumeration formula while at the same time no easy proof of this formula is known. Mills, Robbins and Rumsey \cite{MillsRobbinsRumseyASM} introduced ASMs in the course of generalizing the determinant and conjectured that the number of $n \times n$ ASMs is given by 
\begin{equation}
\label{asmnumbers}
\prod_{j=0}^{n-1} \frac{(3j+1)!}{(n+j)!}.
\end{equation}
More than ten years later, Zeilberger \cite{ZeilbergerASMProof} finally proved their conjecture. Soon after, Kuperberg \cite{KuperbergASMProof} gave another, shorter proof which makes use of a connection to statistical physics where ASMs have appeared before in an equivalent form as a  model for plane square ice (\emph{six vertex model}). Subsequently, it turned out that also many symmetry classes of ASMs can be enumerated by a simple product formula; a majority of the cases were dealt with in \cite{KuperbergSymmClasses}. A standard tool to prove these results are determinantal expressions for the partition function of the six vertex model. A beautiful account on the history of ASMs is provided by Bressoud \cite{BressoudPandC}.

Since an ASM has precisely one $1$ in its first row, it is natural to ask for the number of ASMs where this $1$ is in a prescribed column. Indeed, it turned out that also this refined enumeration leads to a simple product formula \cite{ZeilbergerRefinedASMProof}. Hence, it is also interesting to explore refined enumerations of symmetry classes of ASMs. The task of this paper is to present our attempt to prove the first author's conjecture \cite{FischerOpFormulaVSASM} on a refined enumeration of vertically symmetric alternating sign matrices. While
we are not yet able to complete our proof, we are able to show how it naturally leads to a conjecture on a much more general multivariate Laurent polynomial identity. Moreover, we present some partial results concerning this conjecture and additional motivation why it is interesting to study the conjecture.

A \emph{Vertically
Symmetric Alternating Sign Matrix} (VSASM) is an ASM which is invariant under
reflection with respect to the vertical symmetry axis. For instance, 
$$
\begin{pmatrix}
0 & 0 & 1 & 0 & 0 \\
1 & 0 & -1 & 0 & 1 \\
0 & 0 & 1 & 0 & 0 \\
0 & 1 & -1 & 1 & 0 \\
0 & 0 & 1 & 0 & 0 \\
\end{pmatrix}
$$
is a VSASM. Since the first row of an ASM contains a unique $1$, it follows that VSASMs can only exist
for odd dimensions. Moreover, the
alternating sign condition and symmetry imply that no $0$ can occur in
the middle column. Thus, the middle column of a VSASM has to be
$(1,-1,1,\ldots,-1,1)^T$.  The fact that the unique $1$ of the first row is always
in the middle column implies that the refined enumeration with respect to the first row is
trivial. However, it follows that the second row contains precisely two $1$s and one $-1$.
Therefore, a possible
refined enumeration of VSASMs is with respect to the unique $1$ in the second row that is situated
left of the middle column. Let $B_{n,i}$ denote the number of $(2n+1)\times
(2n+1)$-VSASMs where the first $1$ in the second row is in
column $i$. In \cite{FischerOpFormulaVSASM}, the first author conjectured that
\begin{equation}
\label{bniFormula}
B_{n,i}= \frac{\binom{2n+i-2}{2n-1} \binom{4n-i-1}{2n-1}}{\binom{4n-2}{2n-1}} 
 \prod_{j=1}^{n-1} \frac{(3j-1)(2j-1)!(6j-3)!}{(4j-2)!(4j-1)!}, \quad
i=1,\ldots,n.
\end{equation}
Let us remark that another possible refined enumeration is the one with respect to the first
column's unique $1$. Let $B_{n,i}^*$ denote the number of VSASMs of size
$2n+1$ where the first column's unique $1$ is located in row $i$. In \cite{RazStroVSASM},
A. Razumov and Y. Stroganov showed that 
\begin{equation}
\label{razstro}
B_{n,i}^* =  \prod_{j=1}^{n-1} \frac{(3j-1)(2j-1)!(6j-3)!}{(4j-2)!(4j-1)!} \sum_{r=1}^{i-1}
(-1)^{i+r-1}
\frac{\binom{2n+r-2}{2n-1} \binom{4n-r-1}{2n-1}}{\binom{4n-2}{2n-1}}, \quad i=1,\ldots,2n+1.
\end{equation}
Interestingly, the conjectured formula \eqref{bniFormula} would also imply a particularly simple
linear relation between the two refined enumerations, namely
\begin{equation*}
B_{n,i} = B_{n,i}^* + B_{n,i+1}^*, \quad i=1,\ldots,n.
\end{equation*}
To give a bijective proof of this relation is an open problem. Such a proof would also imply \eqref{bniFormula}.

Our approach is similar to the one used in the proof of the Refined Alternating Sign Matrix Theorem provided by the first author in \cite{FischerNewRefProof}. We summarize some relevant facts from there: Let $A_{n,i}$ denote the number of $n \times n$ ASMs where the unique $1$ in the first row is in column $i$. It was shown that 
$(A_{n,i})_{1 \le i \le n}$ is a solution of the following linear equation system (LES):
\begin{equation}
\label{LESord}
\begin{aligned}
A_{n,i} & = \sum_{j=i}^{n} \binom{2n-i-1}{j-i} (-1)^{j+n} A_{n,j},  &\quad i=1,\ldots,n, \\
A_{n,i} &= A_{n,n+1-i},  &\quad i=1,\ldots,n.
\end{aligned}
\end{equation}
Moreover it was proven that the solution space of this system is one-dimensional. The LES together with the recursion
\begin{equation}
\label{nminus}
A_{n,1} = \sum_{i=1}^{n-1} A_{n-1,i}
\end{equation}
enabled the first author to prove the formula for $A_{n,i}$ by induction with respect to $n$. 

The research presented in this paper started after observing that the numbers $B_{n,i}$ seem to be a solution of a similar LES: 
\begin{equation}
\label{LES}
\begin{aligned}
B_{n,n-i} & = \sum_{j=i}^{n-1} \binom{3n-i-2}{j-i} (-1)^{j+n+1} B_{n,n-j}, &\quad i=-n,-n+1,\ldots,n-1, \\
B_{n,n-i} & = B_{n,n+i+1},  &\quad i=-n,-n+1,\ldots,n-1.
\end{aligned}
\end{equation}
Here we have to be a bit more precise: $B_{n,i}$ is not yet defined if $i=n+1,n+2,\ldots,2n$. However, if we use for the moment 
\eqref{bniFormula} to define $B_{n,i}$ for all $i \in \mathbb{Z}$, basic hypergeometric manipulations (in fact, only the Chu-Vandermonde summation is involved) imply that 
$(B_{n,i})_{1 \le i \le 2n}$ is a solution of  \eqref{LES}; in Proposition~\ref{lemma:mni2UniqueLES} we show that the solution space of this LES is also one-dimensional. Coming back to the combinatorial definition of $B_{n,i}$, 
the goal of this paper is to show how to naturally extend the combinatorial interpretation of $B_{n,i}$ to $i=n+1,\ldots,2n$ and to present a conjecture of a completely different flavor which, once it is proven, implies that the numbers are a solution of the LES. The identity analogous to \eqref{nminus} is  
$$
B_{n,1} = \sum_{i=1}^{n-1} B_{n-1,i}.
$$
The Chu-Vandermonde summation implies that also the numbers on the right-hand side of \eqref{bniFormula} fulfill this identity, and, once the conjecture presented next is proven, \eqref{bniFormula} also follows by induction with respect to $n$.

In order to be able to formulate the conjecture, we recall that the unnormalized symmetrizer $\sym$ is defined as 
$
\sym p(x_1,\ldots,x_n) :=  \sum\limits_{\sigma \in {\mathcal S}_{n}} p(x_{\sigma(1)},\ldots,x_{\sigma(n)})
$.
\begin{conjecture}
\label{conj}
For integers $s,t \ge 0$, consider the following rational function in $z_1,\ldots,z_{s+t-1}$
$$
P_{s,t}(z_1,\ldots,z_{s+t-1}):= 
\prod_{i=1}^{s} z_i^{2s-2i-t+1} (1-z_i^{-1})^{i-1}
\prod_{i=s+1}^{s+t-1} z_i^{2i-2s-t} (1-z_i^{-1})^s \prod_{1 \le p < q \le s+t-1} 
\frac{1 - z_p + z_p z_q}{z_q - z_p}
$$
and let 
$R_{s,t}(z_1,\ldots,z_{s+t-1}) := \sym P_{s,t}(z_1,\ldots,z_{s+t-1})$.
If $s \le t$, then 
$$
R_{s,t}(z_1,\ldots,z_{s+t-1}) = R_{s,t}(z_1,\ldots,z_{i-1},z_i^{-1},z_{i+1},\ldots,z_{s+t-1})
$$
for all $i \in \{1,2,\ldots,s+t-1\}$.
 \end{conjecture}

Note that in fact the following more general statement seems to be true: if $s \le t$, then 
$ R_{s,t}(z_1,\ldots,z_{s+t-1})$
is a linear combination of expressions of the form
$\prod\limits_{j=1}^{s+t-1} [(z_j -1)(1-z_j^{-1})]^{i_j}$, $i_j \ge 0$, where the coefficients are non-negative integers.
Moreover, it should be mentioned that it is easy to see that 
$R_{s,t}(z_1,\ldots,z_{s+t-1})$ is in fact a Laurent polynomial: Observe that 
$$
R_{s,t}(z_1,\ldots,z_{s+t-1}) =  
\frac{\asym \left(P_{s,t}(z_1,\ldots,z_{s+t-1}) \prod\limits_{1 \le i < j \le s+t-1} (z_j-z_i) \right)}{\prod\limits_{1 \le i < j \le s+t-1} (z_j - z_i)}
$$
with the unnormalized antisymmetrizer $\asym p(x_1,\ldots,x_n) := \sum\limits_{\sigma \in {\mathcal S}_{n}} \sgn \sigma \,p(x_{\sigma(1)},\ldots,x_{\sigma(n)})$. The assertion follows since $P_{s,t}(z_1,\ldots,z_{s+t-1}) \prod\limits_{1 \le i < j \le s+t-1} (z_j-z_i)$ is a Laurent polynomial and every antisymmetric Laurent polynomial is divisible by $\prod\limits_{1 \le i < j \le s+t-1} (z_j - z_i)$.

We will prove the following two theorems.

\begin{theorem} 
\label{implies}
Let $R_{s,t}(z_1,\ldots,z_{s+t-1})$ be as in Conjecture~\ref{conj}. If 
$$
R_{s,t}(z_1,\ldots,z_{s+t-1}) = R_{s,t}(z_1^{-1},\ldots,z_{s+t-1}^{-1})
$$
for all $1 \le s \le t$, then \eqref{bniFormula} is fulfilled.
\end{theorem}

\begin{theorem}
\label{inductionstep}
Let $R_{s,t}(z_1,\ldots,z_{s+t-1})$ be as in Conjecture~\ref{conj}. Suppose
\begin{equation}
\label{allinverse}
R_{s,t}(z_1,\ldots,z_{s+t-1}) = R_{s,t}(z_1^{-1},\ldots,z_{s+t-1}^{-1})
\end{equation}
if $t=s$ and $t=s+1$, $s \ge 1$. Then 
\eqref{allinverse} holds for all $s,t$ with $1 \le s \le t$. 
\end{theorem}

While we believe that \eqref{bniFormula} should probably be attacked with the six vertex model approach (although we have not tried), we also think that the more general Conjecture~\ref{conj} is interesting in its own right, given the fact that it only involves very elementary mathematical objects such as rational functions and the symmetric group.

The paper is organized as follows. 
We start by showing that the solution space of \eqref{LES} is one-dimensional. 
Then we provide a first expression for $B_{n,i}$ and present linear equation systems that generalize 
the system in the first line of \eqref{LESord} and the system in the first line of \eqref{LES} when restricting to non-negative 
$i$ in the latter. Next we use the expression for $B_{n,i}$ to extend the combinatorial interpretation to $i=n+1,n+2,\ldots,2n$ and also extend the linear equation system to negative integers $i$ accordingly. In Section~\ref{seq}, we justify the choice of certain constants that are involved in this extension.
Afterwards we present a first conjecture implying \eqref{bniFormula}. Finally, we are able to prove Theorem~\ref{implies}. The proof of Theorem~\ref{inductionstep} is given in Section~\ref{Theorem2}. It is independent of the rest of the paper and, at least for our taste, quite elegant.  We would love to see a proof of Conjecture~\ref{conj} which is possibly along these lines. 
We conclude with some remarks concerning the special $s=0$ in Conjecture~\ref{conj}, also providing additional motivation why it is of interest to study these symmetrized functions.

\section{The solution space of \texorpdfstring{\eqref{LES}}{(1.6)} is one-dimensional}
\label{dim}

The goal of this section is the proof of the proposition below. Let us remark that we use the following extension of the
binomial coefficient in this paper 
\begin{equation}
\label{bcPolExt}
\binom{x}{j}:=
\begin{cases} 
\frac{x(x-1)\cdots(x-j+1)}{j!}  & \text{if } j \ge 0, \\
0 & \text{if } j < 0,
\end{cases}
\end{equation}
where $x \in \mathbb{C}$ and $j \in \mathbb{Z}$.

\begin{proposition}
\label{lemma:mni2UniqueLES}
For fixed $n \geq 1$, the solution space of the following LES 
\begin{align*}
Y_{n,i} &= \sum_{j=i}^{n-1} \binom{3n-i-2}{j-i} (-1)^{j+n+1} Y_{n,j}, &i=-n,-n+1,\ldots,n-1, \\
Y_{n,i} &= Y_{n,-i-1}, &i=-n,-n+1\ldots,n-1, \\
\end{align*}
in the variables $(Y_{n,i})_{-n \le i \le n-1}$ is one-dimensional.
\end{proposition}
\begin{proof}
As mentioned before, the numbers on the right-hand side of \eqref{bniFormula} are defined for all $i \in \mathbb{Z}$ and establish a solution after replacing $i$ by $n-i$. 
This implies that the solution space is at least one-dimensional. Since
\[
Y_{n,i} = \sum_{j=-n}^{n-1} \binom{3n-i-2}{j-i} (-1)^{j+n+1} Y_{n,-j-1} = \sum_{j=-n}^{n-1} \binom{3n-i-2}{-j-i-1} (-1)^{j+n} Y_{n,j} 
\]
it suffices to show that the $1$-eigenspace of
\[
\left( \binom{3n-i-2}{-j-i-1}(-1)^{j+n} \right)_{-n \leq i,j \leq n-1}
\]
is $1$-dimensional. So, we have to show that
\[
\rk\left( \binom{4n-i-1}{2n-i-j+1}(-1)^{j+1} -\delta_{i,j} \right)_{1 \leq i,j \leq 2n} = 2n-1.
\]
After removing the first row and column and multiplying each row with $-1$, we are done as soon as we show that 
\[
\det \left( \binom{4n-i-1}{2n-i-j+1}(-1)^{j}+\delta_{i,j} \right)_{2 \leq i,j \leq 2n} \neq 0.
\]
If $n=1$, this can be checked directly. Otherwise, 
it was shown in \cite[p.$262$]{FischerNewRefProof} that 
\[
\det \left( \binom{2m-i-1}{m-i-j+1}(-1)^{j}+\delta_{i,j} \right)_{2 \leq i,j \leq m} = \det \left( \binom{i+j}{j-1}+\delta_{i,j} \right)_{1 \leq i,j \leq m-2}
\]
when $m \ge 3$, whereby the last determinant counts descending plane partitions with no part greater than $m-1$, see \cite{AndrewsMacdonald}. However, this number is given by \eqref{asmnumbers} if we set $n=m-1$ there.
\end{proof}

\section{Monotone triangles and an expression for \texorpdfstring{$B_{n,i}$}{B\_\{n,i\}}}
\label{bniexpr}

A \emph{Monotone Triangle} (MT) of size $n$ is
a triangular array of integers $(a_{i,j})_{1 \leq j \leq i \leq n}$, often arranged as follows
\[
\begin{array}{ccccccc}
&&& a_{1,1} \\
&& a_{2,1} && a_{2,2} \\
& \rotatebox{75}{$\ddots$} &&&& \ddots \\
a_{n,1} &&\cdots&  &\cdots&& a_{n,n}
\end{array},
\]
with strict increase along rows, i.e. $a_{i,j} < a_{i,j+1}$, and weak increase along North-East- and
South-East-diagonals, i.e. $a_{i+1,j} \leq a_{i,j} \leq a_{i+1,j+1}$. It is
well-known \cite{MillsRobbinsRumseyASM} that MTs with $n$ rows and bottom row
$(1,2\ldots,n)$ are in one-to-one correspondence with ASMs of size $n$: the
$i$-th row of the MT contains an entry $j$ if the first $i$ rows of the $j$-th column in the corresponding ASM sum up to $1$. 

In order to see that $(2n+1) \times (2n+1)$ VSASMs correspond to MTs with bottom row 
$(2,4,\ldots,2n)$, rotate the VSASM by $90$ degrees. The $(n+1)$-st row
of the rotated VSASM is $(1,-1,1,\ldots,-1,1)$. From the
definition of ASMs, it follows that the vector of partial column sums of the
first $n$ rows is $(0,1,0,\ldots,1,0)$ in this case, i.e. the $n$-th row of the corresponding MT
is $(2,4,\ldots,2n)$. Since the rotated VSASM
is uniquely determined by its first $n$ rows, this establishes a one-to-one
correspondence between VSASMs of size $2n+1$ and MTs with bottom row
$(2,4,\ldots,2n)$. An example of the upper part of a rotated VSASM and its corresponding MT is depicted in Figure \ref{vsasm-stats}.

The refined enumeration of VSASMs directly translates into a refined enumeration of MTs with bottom row $(2,4,\ldots,2n)$: from the
correspondence it follows that $B_{n,i}$ counts MTs
with bottom row $(2,4,\ldots,2n)$ and exactly $n+1-i$ entries equal to $2$ in the
left-most North-East-diagonal (see Figure \ref{vsasm-stats}).
\begin{figure}
\begin{center}
$
\begin{array}{ccc}
\begin{pmatrix}
0 & 0 & 0 & 1 & 0 & 0 & 0 & 0 & 0 \\
0 & \boldsymbol{1} & 0 & -1 & 0 & 1 & 0 & 0 & 0 \\
0 & 0 & 0 & 0 & 0 & 0 & 1 & 0 & 0 \\
0 & 0 & 0 & 1 & 0 & 0 & -1 & 1 & 0 \\
1 & -1 & 1 & -1 & 1 & -1 & 1 & -1 & 1 \\
\end{pmatrix}
&
\Leftrightarrow
&
\begin{array}{ccccccccccc}
&&& 4 \\
&& \boldsymbol{2} && 6 \\
& \boldsymbol{2} && 6 && 7 \\
 \boldsymbol{2} && 4 && 6 && 8 \\
\end{array}
\end{array}
$
\end{center}
\caption{Upper part of a rotated VSASM and its corresponding Monotone
Triangle.\label{vsasm-stats}}
\end{figure}

The problem of counting MTs with fixed bottom row $(k_1,\ldots,k_n)$ was considered in \cite{FischerNumberOfMT}. For each $n\geq 1$, an explicit polynomial $\alpha(n;k_1,\ldots,k_n)$ of degree $n-1$ in each of the $n$ variables $k_1,\ldots,k_n$ was provided such that the evaluation at strictly increasing integers $k_1 < k_2 < \cdots < k_n$ is equal to the number of MTs with fixed bottom row $(k_1,\ldots,k_n)$ -- for instance $\alpha(3;1,2,3)=7$.  In \cite{FischerTopBottom}, it was described how to use the polynomial $\alpha(n;k_1,\ldots,k_n)$ to compute the number of MTs with given bottom row and a certain number of fixed entries in the left-most NE-diagonal:
Let 
\begin{align*}
E_x f(x) &:= f(x+1), \\
\Delta_x f(x) &:= (E_x - \id) f(x) = f(x+1)-f(x), \\
\delta_x f(x) &:= (\id - E_x^{-1}) f(x) = f(x)-f(x-1)
\end{align*}
denote the shift operator and the difference operators.
Suppose  $k_1 \le  k_2 < \cdots < k_n$ and $i \ge 0$, then 
$$(-1)^i \Delta_{k_1}^i \alpha(n;k_1,\ldots,k_n)$$ 
is the number of MTs with bottom
row $(k_1-1,k_2,\ldots,k_n)$ where precisely $i+1$ entries in the left-most
NE-diagonal are equal to $k_1-1$ (see Figure \ref{fig:DeltaiK1}). There exists an analogous result for the 
right-most SE-diagonal: if $k_1 <  \cdots < k_{n-1} \le k_n$, then
$$\delta_{k_n}^i \alpha(n;k_1,\ldots,k_n)$$ 
is the number of MTs where precisely
$i+1$ entries in the right-most SE-diagonal are equal to $k_n+1$ (see Figure
\ref{fig:deltaiKn}).
\begin{figure}[ht]
\begin{minipage}[b]{0.45\linewidth}
\centering
\includegraphics[width=\textwidth]{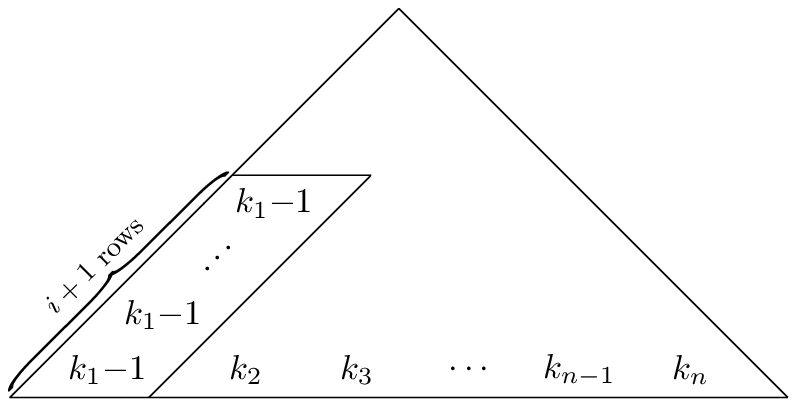}
\caption{$(-1)^i \Delta_{k_1}^i \alpha(n;k_1,\ldots,k_n)$\label{fig:DeltaiK1}}
\end{minipage}
\hspace{0.5cm}
\begin{minipage}[b]{0.45\linewidth}
\centering
\includegraphics[width=\textwidth]{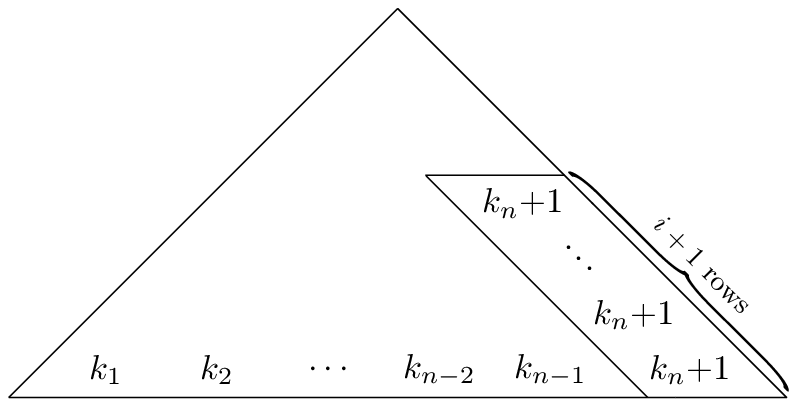}
\caption{$\delta_{k_n}^i \alpha(n;k_1,\ldots,k_n)$\label{fig:deltaiKn}}
\end{minipage}
\end{figure}

This implies the following formula
$$
B_{n,n-i}=(-1)^{i} \Delta_{k_1}^{i} \alpha(n; k_1,4,6,\ldots,2n)|_{k_1=3}.
$$
Let us generalize this by defining
\[
\mn{i}:=(-1)^{i}\Delta_{k_1}^{i}
\alpha(n;k_1,2d,3d,\ldots,nd)|_{k_1=d+1}, \quad d \in \mathbb{Z},\; i \ge 0,
\]
which is for $d \ge 1$ the number of MTs with bottom row
$(d,2d,3d,\ldots,nd)$ and exactly $i+1$ entries equal to $d$ in the left-most
NE-diagonal. If $d=2$, we obtain $B_{n,n-i}$, and it is also not hard to see that we obtain 
the ordinary refined enumeration numbers $A_{n,i+1}$ if $d=1$. Next we prove that the numbers 
$\mn{i}$ fulfill a certain LES. For $d=1$, this proves the first line of \eqref{LESord}, while for $d=2$ it proves the first line of 
\eqref{LES} for non-negative $i$.

\begin{proposition}
\label{lemma:mniLES}
For fixed $n,d \geq 1$ the numbers $(\mn{i})_{0 \le i \le n-1}$ satisfy the following LES 
\begin{equation}
\label{eq:mniLES}
\mn{i}=\sum_{j=i}^{n-1} \binom{n(d+1)-i-2}{j-i} (-1)^{j+n+1} \mn{j}, \quad
i=0,\ldots,n-1.
\end{equation}
\end{proposition}

\begin{proof}
The main ingredients of the proof are the identities 
\begin{align}
\label{eq:alphaPropCyc}
  \alpha(n;k_1,k_2,\ldots,k_n) &= (-1)^{n-1} \alpha(n; k_2,k_3,\ldots,k_n,k_1-n),\\
\label{eq:alphaPropShift}
  \alpha(n;k_1,k_2,\ldots,k_n) &= \alpha(n;k_1+c,k_2+c,\ldots,k_n+c), \quad c \in \mathbb{Z}.
\end{align}
A proof of the first identity was given in \cite{FischerNewRefProof}. The second identity is obvious by combinatorial arguments if $k_1 < k_2 < \cdots < k_n$ and is therefore
also true as identity satisfied by the polynomial. Together with $\Delta_x =
E_x \delta_x$, $E^{-1}_x=(\id - \delta_x)$ and the Binomial Theorem we obtain
\begin{align*}
\mn{i}&= (-1)^{i}\Delta_{k_1}^{i}
\alpha(n;k_1,2d,3d,\ldots,nd)|_{k_1=d+1} \\
&= (-1)^{i+n+1} \Delta_{k_1}^{i} \alpha(n;2d,3d,\ldots,nd,k_1-n)|_{k_1=d+1} \\
&= (-1)^{i+n+1} E_{k_1}^{-n-nd+i+2} \delta_{k_1}^{i}
\alpha(n;2d,3d,\ldots,nd,k_1+d)|_{k_1=nd-1} \\
&= (-1)^{i+n+1} (\id-\delta_{k_1})^{n(d+1)-i-2} \delta_{k_1}^{i}
\alpha(n;d,2d,\ldots,(n-1)d,k_1)|_{k_1=nd-1} \\
&= \sum_{j \geq 0} \binom{n(d+1)-i-2}{j}(-1)^{i+j+n+1} \delta_{k_1}^{i+j}
\alpha(n;d,2d,\ldots,(n-1)d,k_1)|_{k_1=nd-1} \\
&= \sum_{j \geq i} \binom{n(d+1)-i-2}{j-i}(-1)^{j+n+1} \delta_{k_1}^{j}
\alpha(n;d,2d,\ldots,(n-1)d,k_1)|_{k_1=nd-1}. \\
\end{align*}
Since applying the $\delta$-operator to a polynomial decreases its degree, and
$\alpha(n;k_1,\ldots,k_n)$ is a polynomial of degree $n-1$ in each $k_i$, it
follows that the summands of the last sum are zero whenever $j\geq n$. So, it remains to show that 
\begin{equation}
\label{eq:mniAltDesc}
\mn{j} = \delta_{k_1}^{j} \alpha(n;d,2d,\ldots,(n-1)d,k_1)|_{k_1=nd-1}.
\end{equation}
From the discussion preceding the proposition we know that the right-hand side of
\eqref{eq:mniAltDesc} is the number of MTs with bottom row $(d,2d,\ldots,nd)$
and exactly $j+1$ entries equal to $nd$ in the right-most SE-diagonal.
Replacing each entry $x$ of the MT by $(n+1)d-x$ and reflecting
it along the vertical symmetry axis gives a one-to-one correspondence with the objects counted by $\mn{j}$.
\end{proof}

\section{The numbers \texorpdfstring{$\mn{i}$}{C\_\{n,i\}(d)} for \texorpdfstring{$i < 0$}{i < 0}}
\label{extension}
In order to prove \eqref{bniFormula}, it remains to extend the definition of $C_{n,i}^{(2)}$ to $i=-n,\ldots,-1$ in such a way that both the symmetry $C_{n,i}^{(2)} = C_{n,-i-1}^{(2)}$ and the first line of \eqref{LES} is satisfied for negative $i$. Note that the definition of $C_{n,i}^{(2)}$ contains the operator $\Delta_{k_1}^i$ which is per se only defined for $i \geq 0$. The difference operator is (in discrete analogy to differentiation) only invertible up to an additive constant. This motivates the following definitions of right inverse difference operators:

Given a polynomial $p: \mathbb{Z} \to \mathbb{C}$, we define the right inverse difference operators as 
\begin{equation}
\label{def:invDiffOperator}
{^z \Delta^{-1}_x} p(x) := - \sum_{x'=x}^{z} p(x') \qquad \text{ and }  \qquad {^z \delta^{-1}_x} p(x) := \sum_{x'=z}^{x} p(x')
\end{equation}
where $x, z \in \mathbb{Z}$ and the following extended definition of summation
\begin{equation}
\label{sumExtDef}
\sum_{i=a}^b f(i) := \begin{cases} 0, & \quad b=a-1, \\ -\sum\limits_{i =
b+1}^{a-1} f(i), & \quad b+1 \leq a-1, \end{cases}
\end{equation}
is used. The motivation for the extended definition is that it preserves polynomiality: suppose $p(i)$ is a polynomial in $i$ then
$(a,b) \mapsto \sum\limits_{i=a}^{b} p(i)$ is a polynomial function on $\mz^2$. The following identities can be easily checked.
\begin{proposition} 
\label{prop:inv_identities}
Let $z \in \mathbb{Z}$ and $p: \mathbb{Z} \to \mathbb{C}$ a function. Then
\begin{enumerate}
\item\label{prop:inv_identities1} $\Delta_x \, {^z \Delta^{-1}_x} = \id$ \, and \, ${^z \Delta^{-1}_x} \Delta_x p(x) = p(x) - p(z+1)$,
\item\label{prop:inv_identities2} $\delta_x  \, {^z \delta^{-1}_x} = \id$ \, and \, ${^z \delta^{-1}_x} \delta_x p(x) = p(x) - p(z-1)$,
\item\label{prop:inv_identities3} $\Delta_x = E_x \delta_x$ \, and \, ${^z \Delta^{-1}_x} = E_x^{-1} E_z \, {^z \delta^{-1}_x}$,
\item\label{prop:inv_identities4} $\Delta_y \, {^z \Delta^{-1}_x} = {^z \Delta^{-1}_x} \Delta_y$ and $\delta_y \, {^z \Delta^{-1}_x} = {^z \Delta^{-1}_x} \delta_y$ for $y \neq x,z$.  
\end{enumerate}
\end{proposition}

Now we are in the position to define higher negative powers of the difference operators: 
For $i < 0$ and $\mathbf{z}=(z_i,z_{i+1},\ldots,z_{-1}) \in \mz^{-i}$ we let
\begin{align*}
{^\mathbf{z} \Delta^{i}_x} &:= {^{z_i} \Delta^{-1}_x} \, {^{z_{i+1}} \Delta^{-1}_x} \dots {^{z_{-1}} \Delta^{-1}_x}, \\
{^\mathbf{z} \delta^{i}_x} &:= {^{z_i} \delta^{-1}_x} \, {^{z_{i+1}} \delta^{-1}_x} \dots {^{z_{-1}} \delta^{-1}_x}.
\end{align*}
After observing that ${^z} \delta^{-1}_x E^{-1}_x = E^{-1}_x E^{-1}_z {^z} \delta^{-1}_x$ we can deduce the following generalization of Proposition~\ref{prop:inv_identities} \eqref{prop:inv_identities3} inductively:
\begin{equation}
\label{eq:Deltatodelta}
{^\mathbf{z} \Delta^{i}_x} = 
E_x^{i} E_{z_i}^{i+2} E_{z_{i+1}}^{i+3} \dots E_{z_{-1}}^{1}  
{^\mathbf{z} \delta^{i}_x}.
\end{equation}

The right inverse difference operator allows us to naturally extend the definition of $\mn{i}$: First, let us fix a sequence of integers $\mathbf{x}=(x_j)_{j < 0}$ and set 
$\mathbf{x}_i=(x_i,x_{i+1},\ldots,x_{-1})$ for $i<0$. We define
\begin{equation}
\label{defc}
\mn{i} := 
\begin{cases}
\left. (-1)^{i}\Delta_{k_1}^{i} \alpha(n;k_1,2d,3d,\ldots,nd)\right|_{k_1=d+1}, & i=0,\ldots,n-1, \\
\left. (-1)^{i} \,\, {^{\mathbf{x}_i} \Delta_{k_1}^{i}} \alpha(n;k_1,2d,3d,\ldots,nd)\right|_{k_1=d+1}, & i=-n,\ldots,-1.
\end{cases}
\end{equation}
We detail on the choice of $\mathbf{x}$ in Section~\ref{seq}.

If $d \ge 1$, it is possible to give a rather natural combinatorial interpretation of $\mn{i}$ also for negative $i$ which is based on previous work of the authors. It is of no importance for the rest of the paper, however, it provides a nice intuition: One can show that for non-negative $i$, the quantity $\mn{i}$ counts partial MT where we cut off the bottom $i$ elements of the left-most NE-diagonal, prescribe the entry $d+1$ in position $i+1$ of the NE-diagonal and the entries $2d, 3d, \ldots, n d$ in the bottom row of the remaining array (see Figure \ref{fig:mni2Pos}); in fact, in the exceptional case of  $d=1$ we do not require that the bottom element $2$ of the truncated 
left-most NE-diagonal is strictly smaller than its right neighbor.

From \eqref{def:invDiffOperator} it follows that applying the inverse difference operator has the opposite effect of prolonging the left-most NE-diagonal:
if $i < 0$, the quantity $\mn{i}$ is the \emph{signed} enumeration of arrays of the shape as depicted in Figure~\ref{fig:mni2Neg} subject 
to the following conditions:
\begin{itemize}
\item For the elements in the prolonged NE-diagonal including the entry left of the entry $2d$, we require the following: Suppose $e$ is such 
an element and $l$ is its SW-neighbor and $r$ its SE-neighbor: if $l \le r$, then $l \le e \le r$; otherwise $r < e < l$. In the latter case, the element contributes a 
$-1$ sign.
\item Inside the triangle, we follow the rules of Generalized Monotone Triangles as presented in \cite{RieglerGMT}. The total sign is the product of the sign 
of the Generalized Monotone Triangle and the signs of the elements in the prolonged NE-diagonal.
\end{itemize}
\begin{figure}[ht]
\begin{minipage}[b]{0.45\linewidth}
\centering
\includegraphics[width=\textwidth]{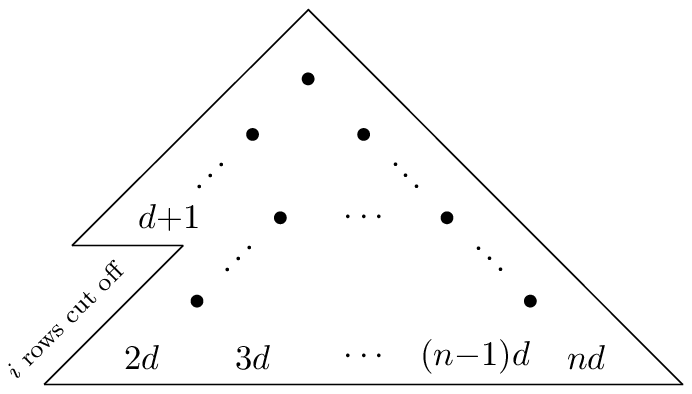}
\caption{$C_{n,i}^{(d)}$ for $i \geq 0$.\label{fig:mni2Pos}}
\end{minipage}
\hspace{0.5cm}
\begin{minipage}[b]{0.45\linewidth}
\centering
\includegraphics[width=\textwidth]{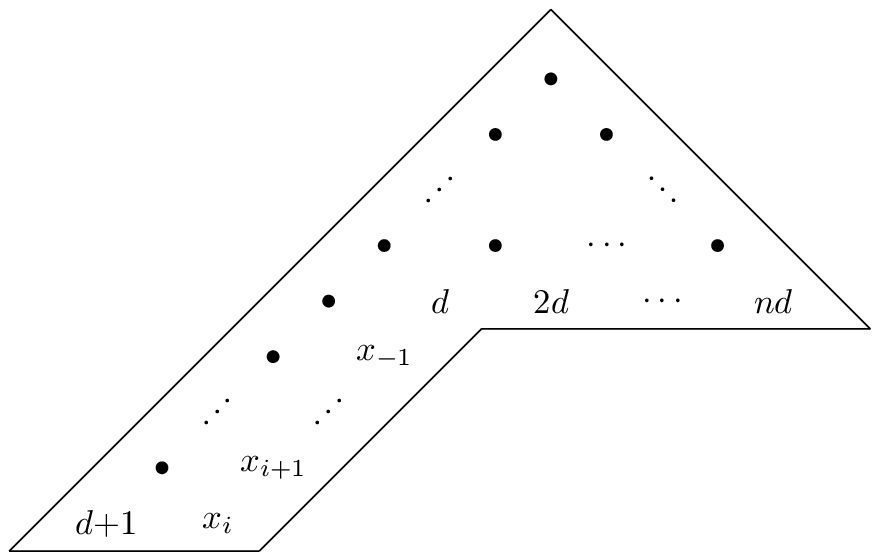}
\caption{$C_{n,i}^{(d)}$ for $i < 0$.\label{fig:mni2Neg}}
\end{minipage}
\end{figure}

\section{Extending the LES to negative \texorpdfstring{$i$}{i}}
\label{properties}

The purpose of this section is the extension of the LES in Proposition~\ref{lemma:mniLES} to negative $i$. This is accomplished with the help of the following lemma which shows that certain identities 
for $\Delta_{k_1}^i \alpha(n;k_1,\ldots,k_n)$, $i \geq 0$, carry over into the world of inverse difference operators.

\begin{lemma}
\label{lemma:reflectcyclic}
Let $n,d\geq 1$.
\begin{enumerate}
\item Suppose $i \ge 0$. Then
$$
\left. (-1)^{i} \,\, {\Delta_{k_1}^{i}} \alpha(n;k_1,2d,3d,\ldots,n d) \right|_{k_1=d+1} \\
= \left. {\delta_{k_n}^{i}} \alpha(n;d,2d,\ldots,(n-1) d,k_n) \right|_{k_n=nd-1}.
$$
\item Suppose $i < 0$, and let $\mathbf{x}_i=(x_i,\ldots,x_{-1})$ and $\mathbf{y}_i=(y_i,\ldots,y_{-1})$ satisfy the relation 
$y_j = (n+1)d - x_j$ for all $j$.
Then (see Figure \ref{fig:DeltaInvSymmetry})
$$
\left. (-1)^{i} \,\, {^{\mathbf{x}_i} \Delta_{k_1}^{i}} \alpha(n;k_1,2d,3d,\ldots,n d) \right|_{k_1=d+1} \\
= \left. {^{\mathbf{y}_i} \delta_{k_n}^{i}} \alpha(n;d,2d,\ldots,(n-1) d,k_n) \right|_{k_n=nd-1}.
$$
\item Suppose $i \ge 0$. Then
$$
\Delta^i_{k_1} \alpha(n;k_1,\ldots,k_n) = (-1)^{n-1} E^{i-n}_{k_1} \delta^i_{k_1} \alpha(n;k_2,\ldots,k_n,k_1).
$$
\item Suppose $i < 0$, and let $\mathbf{x}_i=(x_i,\ldots,x_{-1})$ and $\mathbf{y}_i=(y_i,\ldots,y_{-1})$ satisfy the relation 
$
y_j = x_j+j-n+2$ for all $j$. 
Then
$$
{{^{\mathbf{x}_i}}\Delta^i_{k_1}} \alpha(n;k_1,\ldots,k_n) = (-1)^{n-1} E^{i-n}_{k_1} \,\, {^{\mathbf{y}_i} \delta^i_{k_1}} \alpha(n;k_2,\ldots,k_n,k_1).
$$
\end{enumerate}
\end{lemma}
\begin{figure}[ht]
\begin{minipage}[b]{0.45\linewidth}
\centering
\includegraphics[width=\textwidth]{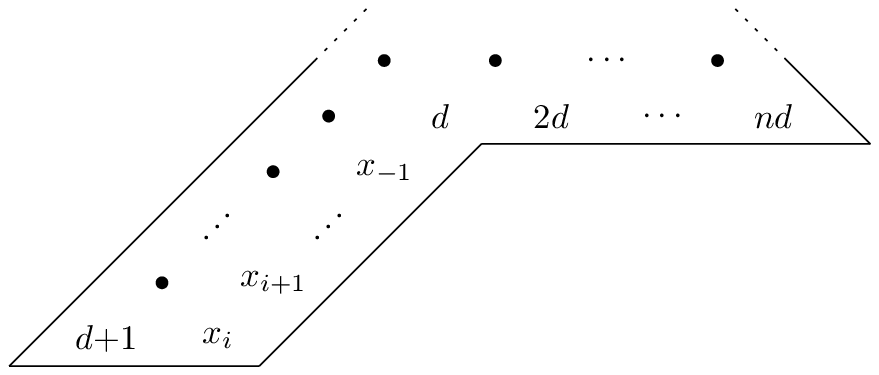}
\end{minipage}
\hspace{0.5cm}
\begin{minipage}[b]{0.45\linewidth}
\centering
\includegraphics[width=\textwidth]{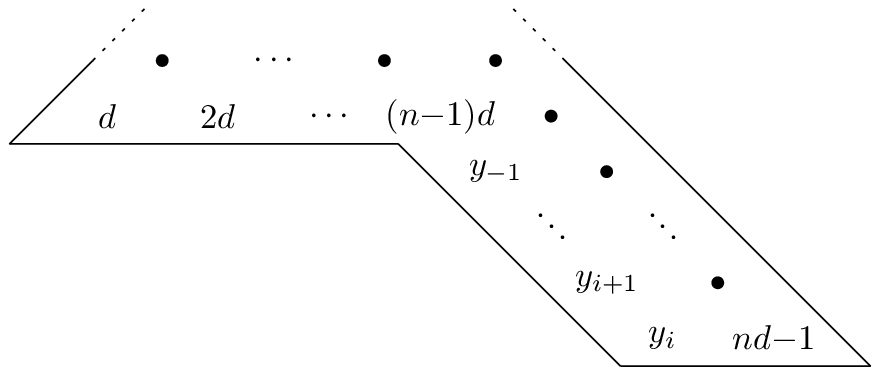}
\end{minipage}
\caption{Symmetry of inverse difference operators if $y_j = (n+1)d - x_j$.\label{fig:DeltaInvSymmetry}}
\end{figure}

\begin{proof}
For the first part we refer to \eqref{eq:mniAltDesc}. Concerning the second part, we actually show the following more general statement: if $r=(n+1)d-l$ and $i \le 0$, then
\begin{equation}
\label{generalize}
\left. (-1)^{i} \,\, {^{\mathbf{x}_i} \Delta_{k_1}^{i}} \alpha(n;k_1,2d,3d,\ldots,n d) \right|_{k_1=l} \\
= \left. {^{\mathbf{y}_i} \delta_{k_n}^{i}} \alpha(n;d,2d,\ldots,(n-1) d,k_n) \right|_{k_n=r}.
\end{equation}
We use induction with respect to $i$; the case $i=0$ is covered by the first part 
(${^{\mathbf{x}_0} \Delta_{k_1}^{0}}=\id={^{\mathbf{y}_0} \delta_{k_n}^{0}}$). If $i<0$, then, by the definitions of the right inverse operators and the induction hypothesis, we have 
\begin{align*}
\left. (-1)^{i} \,\, {^{\mathbf{x}_i} \Delta_{k_1}^{i}} \alpha(n;k_1,2d,3d,\ldots,n d) \right|_{k_1=l} 
&=  \sum_{k'_1=l}^{x_i}  (-1)^{i+1} \,\, {^{\mathbf{x}_{i+1}} \Delta_{k'_1}^{i+1}} \alpha(n;k'_1,2d,3d,\ldots,n d) \\
&=   \sum_{k'_1=l}^{x_i}  \left. {^{\mathbf{y}_{i+1}} \delta_{k'_n}^{i+1}} \alpha(n;d,2d,\ldots,(n-1) d,k'_n) \right|_{k'_n=(n+1)d-k'_1} \\
&= \sum_{k'_n=(n+1)d-x_i}^{(n+1)d-l}  {^{\mathbf{y}_{i+1}} \delta_{k'_n}^{i+1}} \alpha(n;d,2d,\ldots,(n-1) d,k'_n).
\end{align*}
The last expression is equal to the right-hand side of the claimed identity.

The third part follows from \eqref{eq:alphaPropCyc} and Proposition~\ref{prop:inv_identities} \eqref{prop:inv_identities3}. 
The last part is shown by induction with respect to $i$; in fact $i=0$ can be chosen to be the initial case of the induction. If 
$i<0$, then the induction hypothesis and \eqref{sumExtDef} imply
\begin{align*}
{{^{\mathbf{x}_i}}\Delta^i_{k_1}} \alpha(n;k_1,\ldots,k_n) &= - \sum_{l_1=k_1}^{x_i} 
{{^{\mathbf{x}_{i+1}}}\Delta^{i+1}_{l_1}} \alpha(n;l_1,k_2,\ldots,k_n) \\
& = - \sum_{l_1=k_1}^{x_i}  (-1)^{n-1} E^{i+1-n}_{l_1} \,\, {^{\mathbf{y}_{i+1}}\delta^{i+1}_{l_1}} \alpha(n;k_2,\ldots,k_n,l_1) \\
& = \sum_{l_1=x_i+i-n+2}^{k_1+i-n} (-1)^{n-1} \, {^{\mathbf{y}_{i+1}}\delta^{i+1}_{l_1}} \alpha(n;k_2,\ldots,k_n,l_1).
\end{align*}
The last expression is obviously equal to the right-hand side of the identity in the lemma.
\end{proof}

Now we are in the position to generalize Proposition~\ref{lemma:mniLES}.

\begin{proposition}
\label{prop:cniLES}
Let $n,d \geq 1$. For $i<0$, let  $\mathbf{x}_i, \mathbf{z}_i \in \mathbb{Z}^{-i}$ with $z_j =(n+2)(d+1)- x_j-j-4$ and define
\begin{equation}
\label{defd}
D^{(d)}_{n,i} := 
\begin{cases}
\left. (-1)^{i}\Delta_{k_1}^{i} \alpha(n;k_1,2d,3d,\ldots,nd)\right|_{k_1=d+1}, & i=0,\ldots,n-1, \\
\left. (-1)^{i} \,\, {^{\mathbf{z}_i} \Delta_{k_1}^{i}} \alpha(n;k_1,2d,3d,\ldots,nd)\right|_{k_1=d+1}, & i=-n,\ldots,-1.
\end{cases}
\end{equation}
Then  
 $$
\mn{i} = \sum_{j=i}^{n-1}\binom{n(d+1)-i-2}{j-i} (-1)^{j+n+1} D^{(d)}_{n,j} .
$$
holds for all $i=-n,\ldots,n-1$.
\end{proposition}

\begin{proof}
To simplify notation let us define ${^{\mathbf{x}_i} \Delta_{k_1}^{i}} := \Delta_{k_1}^{i}$ for $i \geq 0$.
Since the definition of $C_{n,i}^{(d)}$ and $D_{n,i}^{(d)}$ only differ in the choice of constants, the fact that the system of linear equations is satisfied for $i=0,\ldots,n-1$ is Proposition~\ref{lemma:mniLES}. For $i=-n,\ldots,-1$ first note that, by Lemma~\ref{lemma:reflectcyclic}, \eqref{eq:alphaPropShift} and $E_x^{-1}\, {^{z}}\delta^{-1}_x = {^{z+1}} \delta^{-1}_x E_x^{-1}$, we have
$$
C^{(d)}_{n,i} = \left. (-1)^{n-1+i} E^{i-n}_{k_1} \,\,
{^{\mathbf{y}_i} \delta^i_{k_1}} \alpha(n;d,2d,\ldots,(n-1) d, k_1) \right|_{k_1=1}
$$
where $\mathbf{y}_i= (y_i,\ldots,y_{-1})$ with $y_j = x_j+j+2-n-d$. This is furthermore equal to
$$
 \left. (-1)^{n-1+i} E^{i-n (d+1) +2}_{k_1} \,\,
{^{\mathbf{y}_i} \delta^i_{k_1}} \alpha(n;d,2d,\ldots,(n-1) d, k_1) \right|_{k_1=n d -1}.
$$
Now we use
$$
E^{i-n(d+1)+2}_{k_1} = (\id - \delta_{k_1})^{n(d+1)-i-2} = \sum_{j=0}^{n(d+1)-i-2} \binom{n(d+1)-i-2}{j} (-1)^j \delta^{j}_{k_1}
$$
and Proposition \ref{prop:inv_identities} \eqref{prop:inv_identities2} to obtain
$$
\sum_{j=0}^{n(d+1)-i-2} \binom{n(d+1)-i-2}{j} (-1)^{n-1+i+j} \,\, \left. {^{\mathbf{y}_{i+j}} \delta^{i+j}_{k_1}} \alpha(n;d,2d,\ldots,(n-1) d, k_1) \right|_{k_1=n d -1}.
$$
Since the (ordinary) difference operator applied to a polynomial decreases the degree, the upper summation limit can be changed to $n-1-i$.
Together with Lemma~\ref{lemma:reflectcyclic} this transforms into
\begin{align*}
\sum_{j=i}^{n-1} & \binom{n(d+1)-i-2}{j-i} (-1)^{n-1+j} \,\, \left. {^{\mathbf{y}_{j}} \delta^{j}_{k_1}} \alpha(n;d,2d,\ldots,(n-1) d, k_1) \right|_{k_1=n d -1} \\
&= \sum_{j=i}^{n-1} \binom{n(d+1)-i-2}{j-i} (-1)^{n-1} \,\, \left. {^{\mathbf{z}_{j}} \Delta^{j}_{k_1}} \alpha(n;k_1,2d,3d,\ldots,n d) \right|_{k_1=d+1}.
\end{align*}
\end{proof}

Now it remains to find an integer sequence $(x_j)_{j<0}$ such that 
$C_{n,i}^{(2)} = C_{n,-i-1}^{(2)}$ and 
$C_{n,i}^{(2)} = D_{n,i}^{(2)}$ for negative $i$.

\section{How to choose the sequence \texorpdfstring{$\mathbf{x}=(x_j)_{j < 0}$}{x=(x\_j) for j < 0}}
\label{seq}

In the section, it is shown that $C_{n,i}^{(2)} = C_{n,-i-1}^{(2)}$ if we choose $\mathbf{x}=(x_j)_{j < 0}$ with $x_j=-2j+1$, $j<0$. This can be deduced from the following more general result.

\begin{proposition}  
\label{prop:mniSymmetry} Let $x_j=-2j+1$, $j <0$, and set $\mathbf{x}_i = (x_i,x_{i+1},\ldots,x_{-1})$ for all $i<0$. 
Suppose $p:\mathbb{Z} \to \mathbb{C}$ and let 
$$ c_i :=
\begin{cases}
\left. (-1)^{i} \Delta_y^{i} p(y) \right|_{y=3}, & i \geq 0, \\
\left. (-1)^{i} \,\, {^{\mathbf{x}_i} \Delta^{i}_y} p(y) \right|_{y=3}, & i < 0, \\
\end{cases}
$$
for $i \in \mathbb{Z}$. Then the numbers satisfy the symmetry $c_i = c_{-i-1}$. 
\end{proposition}
\begin{proof}
We may assume $i \geq 0$. Then 
$$
c_i = \left. (-1)^{i} (E_y - \id)^{i} p(y) \right|_{y=3} = \sum_{d_1=3}^{i+3} \binom{i}{d_1-3} (-1)^{d_1+1} p(d_1),
$$
and
\begin{equation}
\label{iterated}
c_{-i-1} = (-1)^{i+1} \,\,{^{\mathbf{x}_{-i-1}} \Delta_{y}^{-i-1}} \left. p(y) \right|_{y=3} 
= \sum_{d_{i+1}=3}^{2i+3} \sum_{d_{i} = d_{i+1}}^{2i+1} \cdots \sum_{d_2 = d_3}^5 \sum_{d_1=d_2}^3 p(d_1).
\end{equation}

The situation is illustrated in Figure \ref{fig:mn-i2}. According to \eqref{sumExtDef}, the iterated sum is the signed summation of $(d_1,d_2,\ldots,d_{i+1}) \in 
\mathbb{Z}^{i+1}$ subject to the following restrictions: We have 
$3 \le d_{i+1} \le 2i+3$, and for $1 \le j \le i$ the restrictions are   
\begin{equation}
\label{cases}
\begin{aligned}
d_{j+1} \le d_j \le 2j+1 \qquad & \text{if } d_{j+1} \le 2j+1,  \\
d_{j+1} > d_j > 2j+1 \qquad & \text{if } d_{j+1} > 2j+1. 
\end{aligned}
\end{equation} 
Note that there is no admissible $(d_1,d_2,\ldots,d_{i+1})$ with $d_{j+1}=2j+2$. 
The sign of $(d_1,d_2,\ldots,d_{i+1})$ is computed as 
$(-1)^{\# \{1 \le j \le i :\, d_{j} > 2j +1\}}$.

\begin{figure}[ht]
\begin{center}
\scalebox{0.8}{
\includegraphics{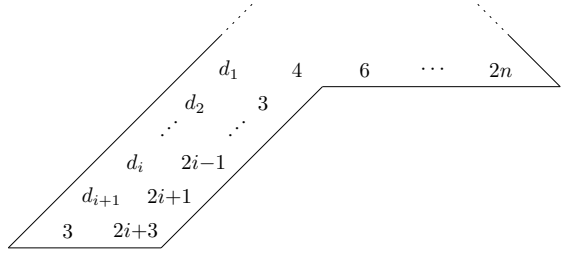}
}
\end{center}
\caption{Combinatorial interpretation of \eqref{iterated} if $p(y)=\alpha(n;y,4,6,\ldots,2n)$.\label{fig:mn-i2}}
\end{figure}

The proof now proceeds by showing that the signed enumeration of $(d_1,\ldots,d_{i+1})$ with fixed $d_1$ is just $\binom{i}{d_1-3} (-1)^{d_1+1}$. The reversed sequence $(d_{i+1},d_{i},\ldots,d_1)$ is weakly increasing as long as we are in the first case of \eqref{cases}. However, once we switch from Case~1 to Case~2, the sequence is strictly decreasing afterwards, because 
$d_{j+1} > 2j+1$ implies $d_{j} > 2j+1 > 2j-1$. Thus, the sequence splits into two parts: there exists an $l$, $0 \le l \le i$, with 
$$
3 \le d_{i+1} \le d_i \le \ldots \le d_{l+1} > d_l > \ldots > d_1.
$$
Moreover, it is not hard to see that \eqref{cases} implies $d_{l+1}=2l+3$ and $d_l=2l+2$. The 
sign of the sequence is $(-1)^l$. Thus it suffices to count the following two types of sequences.
\begin{enumerate}
  \item\label{enum:d1} 
$3 \leq d_{i+1} \leq d_{i} \leq \cdots \leq d_{l+2} \leq d_{l+1} = 2l+3$.
  \item\label{enum:d2} $d_{l} = 2l+2 > d_{l-1} > \cdots > d_2 > d_1 > 3$ and $d_k > 2k+1$ for $1 < k \le l-1$; $d_1$ fixed.
\end{enumerate}
For the first type, this is accomplished by the binomial coefficient $\binom{i+l}{i-l}$.
 
If $l \ge 1$, then the sequences in \eqref{enum:d2} are prefixes of Dyck paths in disguise: to see this, consider prefixes of Dyck paths starting in $(0,0)$ with $a$ steps of type $(1,1)$ and $b$ steps of type $(1,-1)$. Such a partial Dyck path is uniquely determined by the $x$-coordinates of its $(1,1)$-steps. If $p_i$ denotes the position of the $i$-th $(1,1)$-step, then the coordinates correspond to such a partial Dyck path if and only if
\[
0 = p_1 < p_2 < \cdots < p_a < a+b \quad \text{and} \quad p_k < 2k-1. 
\]
In order to obtain \eqref{enum:d2} set $a \mapsto l-1$, $b \mapsto l+3-d_1$ and $p_k \mapsto 2l+2-d_{l-k+1}$. By the reflection principle, the number of prefixes of Dyck paths is 
$$\binom{a+b}{b} \frac{a+1-b}{a+1}=  \binom{2l+2-d_1}{l+3-d_1} \frac{d_1-3}{l}.$$ 

If $l=0$, then $d_1=d_2=\ldots=d_{i+1}=3$ and this is the only case where $d_1=3$. Put together, we see that the coefficient of $p(d_1)$ in \eqref{iterated} is 
\begin{equation}
\label{eq:di}
\sum_{l=1}^{i} (-1)^l \binom{i+l}{i-l} \binom{2l+2-d_1}{l+3-d_1} \frac{d_1-3}{l}
\end{equation}
if $d_1 \ge 4$. 
 Using standard tools to prove hypergeometric identities, it is not hard to see that this is equal to $\binom{i}{d_1-3} (-1)^{d_1+1}$
if $d_1 \ge 4$ and $i \ge 0$.  For instance, C. Krattenthaler's mathematica package HYP \cite{KrattHyp} can be applied as follows: After converting the sum into hypergeometric notation, one applies contiguous relation {\tt C16}. Next we use transformation rule {\tt T4306}, before it is possible to apply summation rule {\tt S2101} which is the Chu-Vandermonde summation.  
\end{proof}

In the following, we let $\mathbf{x}=(x_j)_{j < 0}$ with $x_j=-2j+1$ and $\mathbf{z}=(z_j)_{j < 0}$ with 
$z_j =(n+2)(d+1)+j-5$. Recall that $\mathbf{x}$ is crucial in the definition of $\mn{i}$, see \eqref{defc}, while 
$\mathbf{z}$ appears in the definition of $D^{(d)}_{n,i} $, see \eqref{defd}. To complete the proof of \eqref{bniFormula}, it remains to show 
\begin{equation}
\label{id:openproblem1}
C^{(2)}_{n,i} = D^{(2)}_{n,i}
\end{equation}
for $i=-n,-n+1,\ldots,-1$, since Proposition~\ref{prop:cniLES} and Proposition~\ref{prop:mniSymmetry} then imply that the numbers 
$C^{(2)}_{n,i}$, $i=-n,-n+1,\ldots,n-1$, are a solution of the LES \eqref{LES}. The situation is depicted in Figure
\ref{fig:openProblem}.
\begin{figure}[ht]
\begin{minipage}[b]{0.49\linewidth}
\centering
\includegraphics[width=\textwidth]{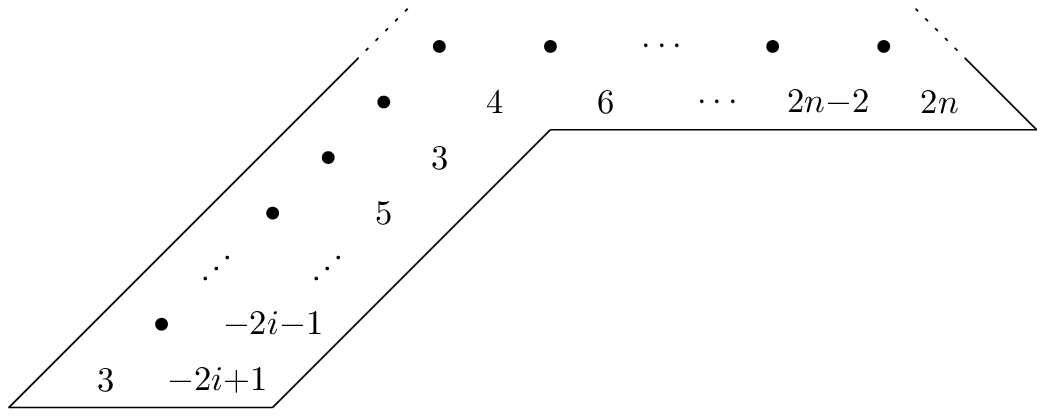}
\end{minipage}
\begin{minipage}[b]{0.49\linewidth}
\centering
\includegraphics[width=\textwidth]{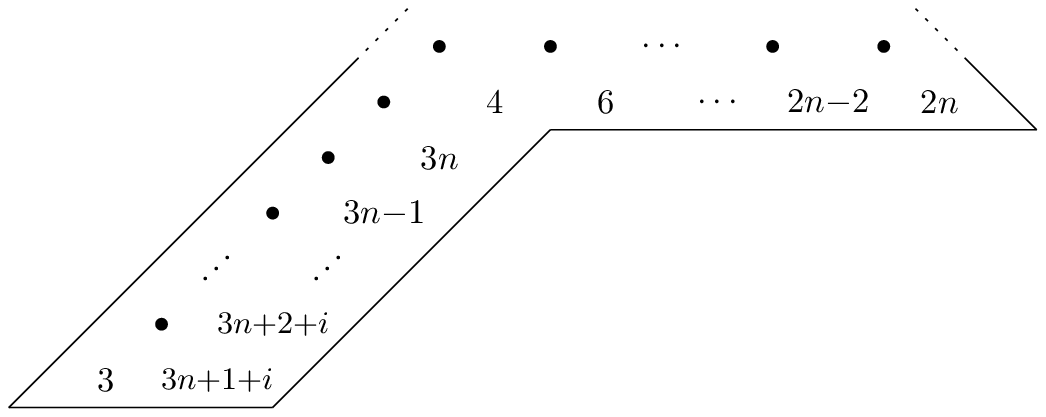}
\end{minipage}
\caption{Combinatorial interpretation of the open problem \eqref{id:openproblem1}.\label{fig:openProblem}}
\end{figure}
When trying to proceed as in the proof of Proposition~\ref{prop:mniSymmetry} one eventually ends up with having to show that the refined VSASM numbers $B_{n,i}$ satisfy a different system of linear equations:
\begin{equation}
\label{eq:newBniLES}
\sum_{j=0}^{n-1} \left( \binom{3n-i-2}{i+j+1}-\binom{3n-i-2}{i-j} \right) (-1)^j B_{n,n-j} = 0, \quad i=0,1,\ldots,n-1.
\end{equation}  
While computer experiments indicate that this LES uniquely determines $(B_{n,1},\ldots,B_{n,n})$ up to a multiplicative constant for all $n \geq 1$, it is not clear at all how to derive that the refined VSASM numbers satisfy \eqref{eq:newBniLES}. We therefore try a different approach in tackling \eqref{id:openproblem1}.

The task of the rest of the paper is to show that \eqref{id:openproblem1} follows from a more general multivariate Laurent polynomial identity and present partial results towards proving the latter.

\section{A first conjecture implying \texorpdfstring{\eqref{id:openproblem1}}{(6.4)}}
\label{firstconj}

We start this section by showing that the application of the right inverse difference operator ${^{z} \Delta}^{-1}_{k_1}$ to $\alpha(n;k_1,\ldots,k_n)$ can be replaced by the application of a bunch of ordinary difference operators to $\alpha(n+1;k_1,z, k_2, \ldots,k_{n})$. Some preparation that already appeared in \cite{FischerNumberOfMT} is needed: The definition of MTs implies (see Figure \ref{fig:mtrec}) that the polynomials $\alpha(n;k_1,\ldots,k_n)$ satisfy the recursion
\begin{equation}
\label{MTRec}
\alpha(n;k_1,\ldots,k_n)= \sum_{\substack{(l_1,\ldots,l_{n-1})\in
\mathbb{Z}^{n-1}, \\
k_1 \leq l_1 \leq k_2 \leq l_2 \leq \cdots \leq k_{n-1} \leq l_{n-1} \leq k_n,
\\ l_i < l_{i+1}}} \alpha(n-1; l_1,\ldots,l_{n-1}),
\end{equation}
whenever $k_1 < k_2 < \cdots < k_n$, $k_i \in \mathbb{Z}$.
\begin{figure}[ht]
\begin{center}
\scalebox{0.8}{
\includegraphics{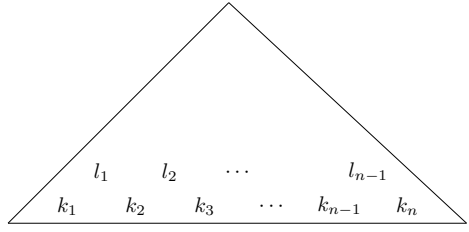}
}
\end{center}
\caption{Bottom and penultimate row of a Monotone Triangle.\label{fig:mtrec}}
\end{figure}
In fact, one can define a summation operator $\sum\limits_{(l_1,\ldots,l_{n-1})}^{(k_1,\ldots,k_n)}$ such that
\begin{equation}
\label{alphaRec}
\alpha(n;k_1,\ldots,k_n) = \sum_{(l_1,\ldots,l_{n-1})}^{\br} \alpha(n-1;l_1,\ldots,l_{n-1})
\end{equation}
for all $(k_1,\ldots,k_n) \in\mathbb{Z}^n$. The postulation that the summation operator should extend \eqref{MTRec} motivates the recursive definition
\begin{align}
\label{sumOpRec}
\sum_{(l_1,\ldots,l_{n-1})}^{\br} A(l_1,\ldots,l_{n-1}) := 
&\sum_{(l_1,\ldots,l_{n-2})}^{(k_1,\ldots,k_{n-1})}
\sum_{l_{n-1} = k_{n-1}+1}^{k_n} A(l_1,\ldots,l_{n-2},l_{n-1}) \\\notag &+
\sum_{(l_1,\ldots,l_{n-2})}^{(k_1,\ldots,k_{n-2},k_{n-1}-1)}
A(l_1,\ldots,l_{n-2},k_{n-1}), \quad n \geq 2
\end{align}
with $\sum\limits_{()}^{(k_1)}:=\id$. Recall the extended definition of the sum over intervals \eqref{sumExtDef} to make sense of this definition for all $(k_1,\ldots,k_n) \in \mathbb{Z}^n$. One can show that this definition ensures that the summation operator preserves
polynomiality, i.e. 
\[
(k_1,\ldots,k_n) \;\mapsto\; \sum_{(l_1,\ldots,l_{n-1})}^{\br} A(l_1,\ldots,l_{n-1})
\]
is a polynomial function on $\mathbb{Z}^n$ whenever $A(l_1,\ldots,l_{n-1})$ is a polynomial. Since a polynomial in $(k_1,\ldots,k_n)$ is uniquely determined by its evaluations at $k_1 < k_2 < \cdots < k_n$, we may also use any other recursive description of the summation operator as long as it is based on the extended definition of ordinary sums \eqref{sumExtDef} and specializes to \eqref{MTRec} whenever $k_1 < k_2 < \cdots < k_n$. So, we can also use the recursive definition
\begin{align}
\label{sumOpRec3}
\sum_{(l_1,\ldots,l_{n-1})}^{\br} A(l_1,\ldots,l_{n-1}) = 
& \sum_{(l_2,\ldots,l_{n-1})}^{(k_2,\ldots,k_{n})} \sum_{l_{1} =
k_1}^{k_2-1} A(l_1,l_2,\ldots,l_{n-1}) \\\notag &+
\sum_{(l_2,\ldots,l_{n-1})}^{(k_2+1,k_3,\ldots,k_n)}
A(k_2,l_2,\ldots,l_{n-1}), \quad n \geq 2.
\end{align}

\begin{lemma} 
\label{lemma:inv_alpha}
Let $i<0$ and $\mathbf{x}_i \in \mathbb{Z}^{-i}$. Then
\begin{multline*}
{^{\mathbf{x}_i}\Delta^{i}_{k_j}} \alpha(n;k_1,\ldots,k_n)= 
(-1)^{i j} \\ 
\times  \Delta_{k_1}^{-i} \ldots \Delta_{k_{j-1}}^{-i} \delta^{0}_{x_{i}} \delta^1_{x_{i+1}} \dots \delta_{x_{-1}}^{-i-1} \delta^{-i}_{k_{j+1}} \dots \delta^{-i}_{k_n} \alpha(n-i;k_1,\ldots,k_j,x_{i},x_{i+1},\ldots,x_{-1},k_{j+1},\ldots,k_n)
\end{multline*}
and 
\begin{multline*}
{^{\mathbf{x}_i} \delta_{k_j}^i} \alpha(n;k_1,\ldots,k_n)= (-1)^{(j-1) i + \binom{-i}{2}} \\
\times \Delta_{k_1}^{-i} \dots \Delta_{k_{j-1}}^{-i} 
\Delta_{x_{-1}}^{-i-1}  \Delta_{x_{-2}}^{-i-2} \dots \Delta_{x_{i}}^{0} 
\delta^{-i}_{k_{j+1}} \dots \delta^{-i}_{k_{n}} \alpha(n-i;k_1,\ldots,k_{j-1},x_{-1},x_{-2},\ldots,x_{i},k_{j},\ldots,k_n).
\end{multline*}
\end{lemma}

\begin{proof}
Informally, the lemma follows from the following two facts:
\begin{itemize}
\item The quantity ${^{\mathbf{x}_i}\Delta^{i}_{k_j}} \alpha(n;k_1,\ldots,k_n)$ can be interpreted as the signed enumeration of Monotone Triangle structures of the shape as depicted in Figure~\ref{fig:DeltaInvIKj} where the $j$-th NE-diagonal has been prolonged. Similarly, for ${^{\mathbf{x}_i}\delta^{i}_{k_j}} \alpha(n;k_1,\ldots,k_n)$, where the shape is depicted in Figure~\ref{fig:sDeltaInvIKj} and the $j$-th SE-diagonal has been prolonged.
\item The application of the $(-\Delta)$-operator truncates left NE-diagonals, while the $\delta$-operator truncates right SE-diagonals. This idea first appeared in \cite{FischerTopBottom}. 
\end{itemize}

\begin{figure}[ht]
\begin{center}
\scalebox{0.8}{
\includegraphics{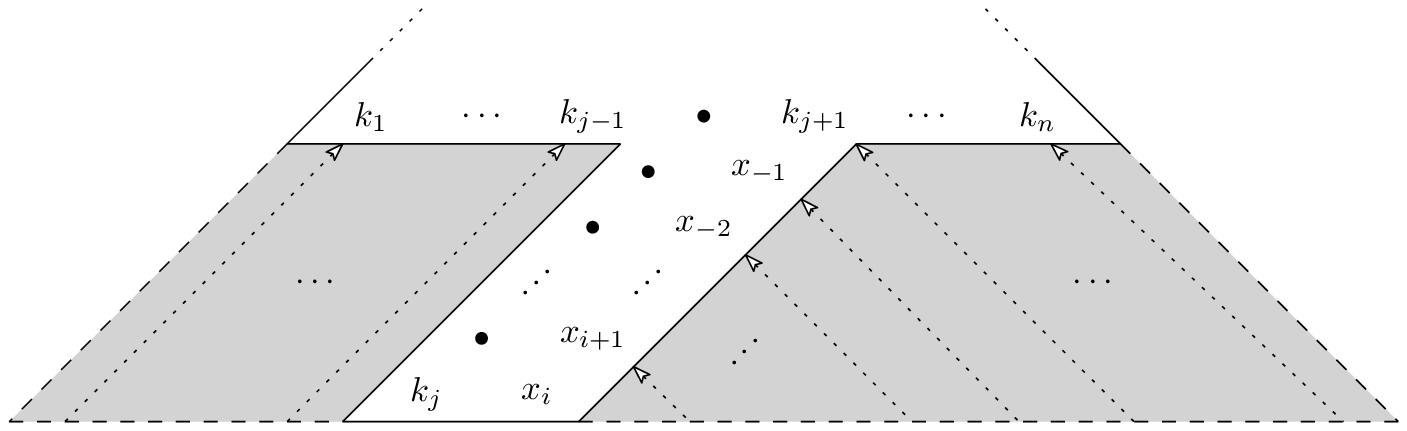}
}
\end{center}
\caption{${^{\mathbf{x}_i}\Delta^{i}_{k_j}} \alpha(n;k_1,\ldots,k_n)$\label{fig:DeltaInvIKj}}
\end{figure}
\begin{figure}[ht]
\begin{center}
\scalebox{0.8}{
\includegraphics{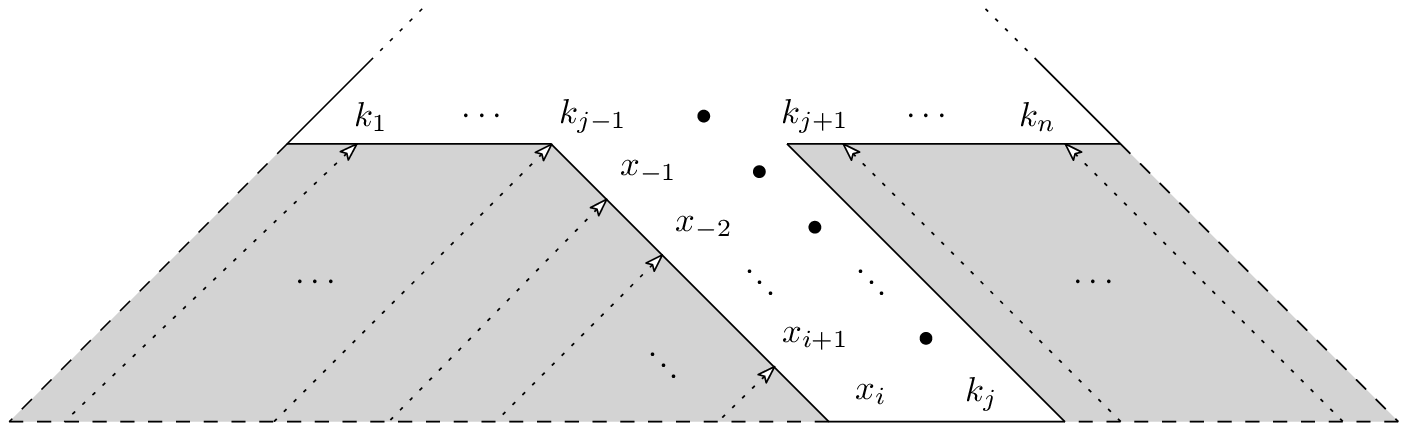}
}
\end{center}
\caption{${^{\mathbf{x}_i}\delta^{i}_{k_j}} \alpha(n;k_1,\ldots,k_n)$\label{fig:sDeltaInvIKj}}
\end{figure}

Formally, let us prove the first identity by induction with respect to $i$. 
First note that \eqref{sumOpRec} and \eqref{sumOpRec3} imply
\begin{align}
\label{eq:invLemmaIng1}
(-1)^j \Delta_{k_1} \dots \Delta_{k_{j-1}} & \delta_{k_{j+1}} \delta_{k_{j+2}} \dots \delta_{k_n} 
\sum_{(l_1,\ldots,l_n)}^{(k_1,\ldots,k_{j-1},k_j,x,k_{j+1},\ldots,k_n)} A(l_1,\ldots,l_n) \\
\notag &= - \sum_{(l_j)}^{(k_j,x)} A(k_1,\ldots,k_{j-1},l_j,k_{j+1},\ldots,k_n) = {^x}\Delta_{k_j}^{-1} A(k_1,k_2,\ldots,k_n).
\end{align}
Together with \eqref{alphaRec} the base case $i=-1$ follows. For the inductive step $i<-1$, apply the induction hypothesis,  \eqref{eq:invLemmaIng1}, \eqref{alphaRec} and Proposition \ref{prop:inv_identities} \eqref{prop:inv_identities4} to obtain
\begin{align*}
& {^{\mathbf{x}_i}\Delta_{k_j}^i} \alpha(n;k_1,\ldots,k_n) \\
&= {^{x_i}\Delta_{k_j}^{-1}} (-1)^{(i+1)j} \Delta_{k_1}^{-i-1} \ldots \Delta_{k_{j-1}}^{-i-1} \delta^{0}_{x_{i+1}} \delta^1_{x_{i+2}} \dots \delta_{x_{-1}}^{-i-2} \delta^{-i-1}_{k_{j+1}} \dots \delta^{-i-1}_{k_n}  \\
& \qquad \alpha(n-i-1;k_1,\ldots,k_j,x_{i+1},x_{i+2},\ldots,x_{-1},k_{j+1},\ldots,k_n) \\
&=  (-1)^{i j} \Delta_{k_1}^{-i} \ldots \Delta_{k_{j-1}}^{-i} \delta^{1}_{x_{i+1}} \delta^2_{x_{i+2}} \dots \delta_{x_{-1}}^{-i-1} \delta^{-i}_{k_{j+1}} \dots \delta^{-i}_{k_n} \\
&\qquad \sum_{(l_1,\ldots,l_{j},y_{i+1},\ldots,y_{-1},l_{j+1},\ldots, l_{n})}^{(k_1,\ldots,k_j,x_i,x_{i+1},\ldots,x_{-1},k_{j+1},\ldots,k_n)} \alpha(n-i-1;l_1,\ldots,l_j,y_{i+1},y_{i+2},\ldots,y_{-1},l_{j+1},\ldots,l_n) \\
&=(-1)^{i j} \Delta_{k_1}^{-i} \ldots \Delta_{k_{j-1}}^{-i} \delta^{1}_{x_{i+1}} \dots \delta_{x_{-1}}^{-i-1} \delta^{-i}_{k_{j+1}} \dots \delta^{-i}_{k_n}  \alpha(n-i;k_1,\ldots,k_j,x_{i},x_{i+1},\ldots,x_{-1},k_{j+1},\ldots,k_n).
\end{align*}

The second identity can be shown analogously. The sign is again obtained by taking the total number of applications of the $\Delta$-operator into account.
\end{proof}

In the following, we let $V_{x,y}:=E^{-1}_x + E_y - E^{-1}_x E_y$ and $S_{x,y}
f(x,y) := f(y,x)$. In \cite{FischerNumberOfMT} it was shown that 
\begin{equation}
\label{eq:antisymalpha}
(\id + E_{k_{i+1}} E^{-1}_{k_i} S_{k_i,k_{i+1}}) V_{k_i,k_{i+1}} \alpha(n;k_1,\ldots, k_n) = 0
\end{equation}
for $1 \le i \le n-1$. This property together with the fact that the degree of  
$\alpha(n;k_1,\ldots, k_n)$ in each $k_i$ is $n-1$ determines 
the polynomial up to a constant. Next we present a conjecture on general polynomials with property \eqref{eq:antisymalpha}; 
the goal of the current section is to show that this conjecture implies \eqref{id:openproblem1}.

\begin{conjecture}
\label{prop:somuchfun}
Let $1 \le s \le t$  and $a(k_1,\ldots,k_{s+t-1})$ be a polynomial in
$(k_1,\ldots,k_{s+t-1})$ with 
\begin{equation}
\label{eq:shiftAntisymmetry}
(\id + E_{k_{i+1}} E^{-1}_{k_i} S_{k_i,k_{i+1}}) V_{k_i,k_{i+1}} a(k_1,\ldots,k_{s+t-1}) = 0
\end{equation}
for $1 \le i \le s+t-2$. Then
\begin{multline*}
\prod_{i=1}^{s} E^{2s+3-2i}_{y_i} \delta^{i-1}_{y_i} \prod_{i=2}^{t} E_{k_i}^{2 i} \delta^s_{k_i}
a(y_1,\ldots,y_s,k_2,\ldots,k_t) \\ =
\prod_{i=2}^t E^{2i}_{k_i} (- \Delta_{k_i})^s \prod_{i=1}^s E^{2t+3-2i}_{y_i} (- \Delta_{y_i})^{s-i}
a(k_2,\ldots,k_t,y_1,\ldots,y_s)
\end{multline*}
if $y_1=y_2=\ldots=y_s=k_2=k_3=\ldots=k_t$. 
\end{conjecture}

\begin{proposition}
\label{cor:mni2cni2}
Let $\mathbf{x}=(-2j+1)_{j<0}$ and $\mathbf{z}=(3n+j+1)_{j<0}$.
Under the assumption that Conjecture~\ref{prop:somuchfun} is true, it follows for all $-n \leq i \leq -1$ that 
\begin{enumerate} 
\item  $\left. {^{\mathbf{x}_i} \Delta_{k_1}^{i}}
\alpha(n;k_1,4,6,\ldots,2n)\right|_{k_1=3n+2+i}=0$,
\item  ${^{\mathbf{x}_i} \Delta_{k_1}^{i}} \alpha(n;k_1,4,6,\ldots,2n) = {^{\mathbf{z}_i} \Delta_{k_1}^{i}} \alpha(n;k_1,4,6,\ldots,2n)$; in particular $C^{(2)}_{n,i} = D^{(2)}_{n,i}$. 
\end{enumerate}
\end{proposition}
\begin{proof}
According to Lemma~\ref{lemma:inv_alpha} we have
\begin{multline*}
{^{\mathbf{x}_i} \Delta_{k_1}^{i}} \alpha(n;k_1,4,6,\ldots,2n) \\
= \left. (-1)^i \prod_{j=i}^{-1} E^{-2j+1}_{y_j} \delta^{j-i}_{y_j} \prod_{j=2}^{n} E^{2 j}_{k_j} \delta^{-i}_{k_j} 
 \alpha(n-i;k_1,y_i,y_{i+1},\ldots,y_{-1},k_2,\ldots,k_n)
 \right|_{\subfl{(y_i,y_{i+1},\ldots,y_{-1})=0,}{(k_2,\ldots,k_n)=0}}.
\end{multline*}
We set $\overline{y}_j=y_{i+j-1}$ and $s=-i$ to obtain
$$
\left. (-1)^s \prod_{j=1}^{s} E^{2s+3-2j}_{\overline{y}_j} \delta^{j-1}_{\overline{y}_j}
\prod_{j=2}^{n} E^{2 j}_{k_j} \delta^{s}_{k_j}
\alpha(n+s;k_1,\overline{y}_1,\overline{y}_2,\ldots,\overline{y}_s,k_2,\ldots,k_n) \right|_{\subfl{(\overline{y}_1,\ldots,\overline{y}_s)=0,}{(k_2,\ldots,k_n)=0}}.
$$
By our assumption that Conjecture~\ref{prop:somuchfun} is true, this is equal to
$$
\left. (-1)^s \prod_{j=2}^n E^{2j}_{k_j} (- \Delta_{k_j})^s \prod_{j=1}^s E^{2n+3-2j}_{\overline{y}_j} (- \Delta_{\overline{y}_j})^{s-j}
 \alpha(n+s;k_1,k_2,\ldots,k_n,\overline{y}_1,\ldots,\overline{y}_s) \right|_{\subfl{(\overline{y}_1,\ldots,\overline{y}_s)=0,}{(k_2,\ldots,k_n)=0}}.
$$
Now we use \eqref{eq:alphaPropCyc} and \eqref{eq:alphaPropShift} to obtain 
\begin{multline*}
(-1)^{n+1} \prod_{j=2}^n E^{2j+n+s}_{k_j} (- \Delta_{k_j})^s \prod_{j=1}^s
E^{3n+3-2j+s}_{\overline{y}_j} (- \Delta_{\overline{y}_j})^{s-j} \\
\left.  \alpha(n+s;k_2,\ldots,k_n,\overline{y}_1,\ldots,\overline{y}_s,k_1)
\vphantom{\prod_{j=2}^n}\right|_{\subfl{(\overline{y}_1,\ldots,\overline{y}_s)=0,}{(k_2,\ldots,k_n)=0}}.
\end{multline*}
According to Lemma~\ref{lemma:inv_alpha}, this is 
$$
(-1)^{n+1} \;{^{\mathbf{w}_i} \delta_{k_1}^i}
\alpha(n;4+n-i,6+n-i,\ldots,3n-i,k_1) $$
where $\mathbf{w}_i=(3n+3+i,3n+5+i,\ldots,3n+1-i)$. Setting $k_1=3n+2+i$, the first assertion
now follows since ${^{x+1} \delta_{x}^{-1}} p(x)=0$.

For the second assertion we use induction with respect to $i$. In the base case
$i=-1$ note that the two sides differ by $\left.  {^{3n}\Delta_{k_1}^{-1}}\alpha(n;k_1,4,6,\ldots,2n)\right|_{k_1=4}$. By \eqref{sumExtDef} this is equal to 
\[
-\left.  {^{3}\Delta_{k_1}^{-1}}\alpha(n;k_1,4,6,\ldots,2n)\right|_{k_1=3n+1},
\]
which vanishes due to the first assertion. For $i < -1$ observe that
\begin{align*}
{^{\mathbf{x}_i} \Delta^i_{k_1}} & \alpha(n;k_1,4,6,\ldots,2n) \\
&= {^{-2i+1} \Delta_{k_1}^{-1}} \;\; {^{\mathbf{x}_{i+1}} \Delta^{i+1}_{k_1}}
\alpha(n;k_1,4,6,\ldots,2n) \\
&= - \sum_{l_1=k_1}^{-2i+1}  {^{\mathbf{x}_{i+1}} \Delta^{i+1}_{l_1}}
\alpha(n;l_1,4,6,\ldots,2n) \\
&=- \sum_{l_1=k_1}^{3n+1+i} {^{\mathbf{z}_{i+1}} \Delta^{i+1}_{l_1}}
\alpha(n;l_1,4,6,\ldots,2n)+\sum_{l_1=-2i+2}^{3n+1+i} {^{\mathbf{x}_{i+1}}
\Delta^{i+1}_{l_1}} \alpha(n;l_1,4,6,\ldots,2n),
\end{align*}
where we have used the induction hypothesis in the first sum. Now the first sum
is equal to the right-hand side in the second assertion, while the second sum is
by \eqref{sumExtDef} just the expression in the first assertion and thus
vanishes.
\end{proof} 

\section{Proof of Theorem \texorpdfstring{\ref{implies}}{1.2}}
\label{Theorem1}

Let $p(x_1,\ldots,x_n)$ be a function in $(x_1,\ldots,x_n)$ and $T \subseteq {\mathcal S}_n$ a subset of the 
symmetric group. We define
$$
(T p)(x_1,\ldots,x_n) := 
\sum_{\sigma \in T} \sgn \sigma 
\, p(x_{\sigma(1)},\ldots,x_{\sigma(n)}).
$$
If $T=\{ \sigma\}$, then we write $( T p)(x_1,\ldots,x_n)=(\sigma p)(x_1,\ldots,x_n)$. Observe that $\asym$ as defined in the introduction satisfies $\asym p(x_1,\ldots,x_n) = ({\mathcal S}_n p)(x_1,\ldots,x_n)$.
A function is said to be \emph{antisymmetric} if 
$
(\sigma p)(x_1,\ldots,x_n) = \sgn\sigma  \cdot p(x_1,\ldots,x_n)
$
for all $\sigma \in {\mathcal S}_n$. We need a couple of auxiliary results.

\begin{lemma}
\label{lemma:operator}
Let $a(z_1,\ldots,z_n)$ be a polynomial in $(z_1,\ldots,z_n)$ with 
$$
(\id + E_{z_{i+1}} E^{-1}_{z_i} S_{z_i,z_{i+1}}) V_{z_i,z_{i+1}} a(z_1,\ldots,z_{n}) = 0
$$
for $1 \le i \le n-1$.  Then there exists an antisymmetric polynomial  
$b(z_1,\ldots,z_n)$ with 
$$
a(z_1,\ldots,z_{n}) = \prod_{1 \le p < q \le n} W_{z_q,z_p} b(z_1,\ldots,z_n)
$$
where $W_{x,y} := E_x V_{x,y} = \id - E_y + E_x E_y$.
\end{lemma}

\begin{proof}
By assumption, we have 
\begin{align*}
S_{z_i,z_{i+1}} W_{z_i,z_{i+1}} a(\mathbf{z}) = E_{z_{i+1}} S_{z_i,z_{i+1}} 
V_{z_i,z_{i+1}} a(\mathbf{z}) = -E_{z_i} V_{z_i,z_{i+1}}a(\mathbf{z}) =
-W_{z_i,z_{i+1}}a(\mathbf{z}).
\end{align*}
This implies that 
$$
c(z_1,\ldots,z_{n}) := \prod_{1 \le p < q \le n} W_{z_p,z_q} a(z_1,\ldots,z_{n})
$$ is an antisymmetric polynomial. Now observe that $W_{x,y}=\id + E_y \Delta_x$ is
invertible on $\mathbb{C}[x,y]$, to be more concrete
$W^{-1}_{x,y} = \sum\limits_{i=0}^{\infty} (-1)^i E_y^{i} \Delta^{i}_x$.
Hence, 
$b(z_1,\ldots,z_{n}) := \prod\limits_{1 \le p \not= q \le n} W^{-1}_{z_p,z_q}
c(z_1,\ldots,z_{n})$
is an antisymmetric polynomial with
$a(z_1,\ldots,z_{n}) = \prod\limits_{1 \le p < q \le n} W_{z_q,z_p} b(z_1,\ldots,z_{n})$.
\end{proof}

\begin{lemma}
\label{lemma:laurentToOp}
Suppose
$
\op(x_1,\ldots,x_n)$ is a Laurent polynomial and $a(z_1,\ldots,z_n)$ is an antisymmetric function. If there exists a non-empty subset $T$ 
of ${\mathcal S}_{n}$ with $(T \op)(x_1,\ldots,x_n)=0$, then
$$
\left. \left( \op(E_{z_1},\ldots,E_{z_n}) a(z_1,\ldots,z_n) \right) \right|_{z_1=z_2=\ldots=z_n}=0.
$$
\end{lemma}

\begin{proof}
First observe that the antisymmetry of $a(z_1,\ldots,z_n)$ implies
$$
(T' a)(z_1,\ldots,z_n) = \sum_{\sigma \in T'} \sgn \sigma a(z_{\sigma(1)},\ldots,z_{\sigma(n)}) = |T'| a(z_1,\ldots,z_n).
$$
for any subset $T' \subseteq {\mathcal S}_n$. Letting
$$
\op(x_1,\ldots,x_n)= \sum_{(i_1,\ldots,i_n) \in \mathbb{Z}^n} c_{i_1,\ldots,i_n} x_1^{i_1} x_2^{i_2} \cdots x_n^{i_n},
$$
we observe that 
\begin{align*}
& \left. \left( \op(E_{z_1},\ldots,E_{z_n}) a(z_1,\ldots,z_n) \right)
\right|_{(z_1,\ldots,z_n)=(d,\ldots,d)} \\
&\qquad=   \sum_{(i_1,\ldots,i_n) \in\mz^n}
c_{i_1,\ldots,i_n} a(i_1+d,\ldots,i_n+d)
=  \frac{1}{|T|} \sum_{(i_1,\ldots,i_n)\in\mz^n}
c_{i_1,\ldots,i_n} (T^{-1} a)(i_1+d,\ldots,i_n+d)
\end{align*}
with $T^{-1} = \{\sigma^{-1} | \sigma \in T\}$, 
since $(i_1,\ldots,i_n) \mapsto a(i_1+d,\ldots,i_n+d)$ is also an antisymmetric function.
This is equal to
\begin{align*}
& \frac{1}{|T|} \sum_{(i_1,\ldots,i_n)\in\mz^n} c_{i_1,\ldots,i_n}
\sum_{\sigma \in T} \sgn\sigma \left. E^{i_{\sigma^{-1}(1)}}_{z_1} \dots
E^{i_{\sigma^{-1}(n)}}_{z_n} a(z_1,\ldots,z_n) \right|_{(z_1,\ldots,z_n)=(d,\ldots,d)} \\
&\qquad= \frac{1}{|T|} \sum_{(i_1,\ldots,i_n)\in\mz^n} c_{i_1,\ldots,i_n}
\sum_{\sigma \in T} \sgn\sigma \left. E^{i_{1}}_{z_{\sigma(1)}} \dots E^{i_{n}}_{z_{\sigma(n)}}
a(z_1,\ldots,z_n) \right|_{(z_1,\ldots,z_n)=(d,\ldots,d)} \\
&\qquad= \left. \frac{1}{|T|} \left[(T \op)(E_{z_1},\ldots,E_{z_n}) \right]
a(z_1,\ldots,z_n) \right|_{(z_1,\ldots,z_n)=(d,\ldots,d)} = 0.
\end{align*}
\end{proof}

Now we are in the position to prove Theorem~\ref{implies}.

\begin{proof}[Proof of Theorem~\ref{implies}]  In order to prove \eqref{bniFormula}, it suffices to show that Conjecture~\ref{prop:somuchfun} holds under the theorem's assumptions.
We set
\begin{align*}
\overline{\op}(z_1,\ldots,z_{s+t-1}) &:= \prod_{i=1}^{s} z_i^{2s+3-2i} (1-z_i^{-1})^{i-1} 
\prod_{i=s+1}^{s+t-1} z_i^{2i-2s+2} (1-z_i^{-1})^s \\
&\qquad - \prod_{i=1}^{t-1} z_i^{2i+2} (1-z_i)^{s} \prod_{i=t}^{s+t-1} z_i^{4t+1-2i} 
(1-z_i)^{s+t-1-i}
\end{align*}
and observe that the claim of Conjecture~\ref{prop:somuchfun} is that $\overline{\op}(E_{z_1},\ldots,E_{z_{s+t-1}}) a(z_1,\ldots,z_{s+t-1})$
vanishes if $z_1=\ldots=z_{s+t-1}$. According to Lemma~\ref{lemma:operator}, there exists an antisymmetric polynomial 
$b(z_1,\ldots,z_{s+t-1})$ with 
$$
a(z_1,\ldots,z_{s+t-1}) = \prod_{1 \le p < q \le s+t-1} W_{z_q,z_p} b(z_1,\ldots,z_{s+t-1}).
$$
Thus, let us deduce that $\op(E_{z_1},\ldots,E_{z_{s+t-1}}) b(z_1,\ldots,z_{s+t-1})=0$ if $z_1=\ldots=z_{s+t-1}$ where 
$$
\op(z_1,\ldots,z_{s+t-1}):= \overline{\op}(z_1,\ldots,z_{s+t-1}) \prod_{1 \leq p < q \leq s+t-1} (1-z_p + z_p z_q) \prod_{i=1}^{s+t-1} z_i^{-2-t}.
$$
Now, Lemma~\ref{lemma:laurentToOp} implies that it suffices to show $\asym \op(z_1,\ldots,z_{s+t-1}) = 0$.
Observe that 
$$
\op(z_1,\ldots,z_{s+t-1}) = \overline{P}_{s,t}(z_1,\ldots,z_{s+t-1}) - \overline{P}_{s,t}(z_{s+t-1}^{-1},\ldots,z_1^{-1})  \prod_{i=1}^{s+t-1} z_i^{s+t-2}
$$
where $\overline{P}_{s,t}(z_1,\ldots,z_{s+t-1}) = P_{s,t}(z_1,\ldots,z_{s+t-1}) \prod\limits_{1 \le i < j \le s+t-1} (z_j - z_i)$ and 
$P_{s,t}(z_1,\ldots,z_{s+t-1})$ is as defined in Conjecture~\ref{conj}.
Furthermore, 
\begin{align*}
\asym \op(z_1,\ldots,z_{s+t-1}) &= R_{s,t}(z_1,\ldots,z_{s+t-1}) \prod_{1 \le i < j \le s+t-1} (z_j-z_i) \\
& \quad - 
R_{s,t}(z_{s+t-1}^{-1},\ldots,z_1^{-1}) \prod_{1 \le i < j \le s+t-1} (z_{s+t-j}^{-1}-z_{s+t-i}^{-1}) \prod_{i=1}^{s+t-1} z_i^{s+t-2}
\end{align*}
where $R_{s,t}(z_1,\ldots,z_{s+t-1})$ is also defined in Conjecture~\ref{conj}. Since $R_{s,t}(z_1,\ldots,z_{s+t-1})$ is symmetric we have that
$\asym \op(z_1,\ldots,z_{s+t-1}) = 0$ follows once it is shown that 
$R_{s,t}(z_1,\ldots,z_{s+t-1}) = R_{s,t}(z_1^{-1},\ldots,z_{s+t-1}^{-1})$.
\end{proof}

\section{Proof of Theorem \texorpdfstring{\ref{inductionstep}}{1.3}}
\label{Theorem2}

For integers $s,t \ge 1$, we define the following two rational functions:
\begin{align*}
S_{s,t}(z;z_1,\ldots,z_{s+t-2})&:= z^{2s-t-1} \prod_{i=1}^{s+t-2} \frac{(1-z+ z_i z)(1-z_i^{-1})}{(z_i-z)}, \\
T_{s,t}(z;z_1,\ldots,z_{s+t-2}) &:= (1-z^{-1})^s z^{t-2} \prod_{i=1}^{s+t-2} \frac{1-z_i + z_i z}{(z-z_i) z_i}.
\end{align*}
Based on these two functions, we define two operators on functions $f$ in $s+t-2$ variables that transform them into functions in $(z_1,\ldots,z_{s+t-1})$:
\begin{align*}
\ps_{s,t}[f]&:= S_{s,t}(z_{1};z_2,\ldots,z_{s+t-1}) \cdot f(z_2,\ldots,z_{s+t-1}), \\
\pt_{s,t}[f]&:=T_{s,t}(z_{s+t-1};z_1,\ldots,z_{s+t-2}) \cdot f(z_1,\ldots,z_{s+t-2}).
\end{align*}
The definitions are motivated by the fact that $P_{s,t}(z_1,\ldots,z_{s+t-1})$ as defined in Conjecture~\ref{conj} satisfies the two recursions
$$
P_{s,t}=\ps_{s,t}[P_{s-1,t}]  \quad \text{ and } \quad P_{s,t} = \pt_{s,t}[P_{s,t-1}].
$$
We also need the following two related operators, which are again defined on functions $f$ in $s+t-2$ variables:
\begin{align*}
\qs_{s,t}[f]&:= S_{s,t}(z_{s+t-1}^{-1};z_{s+t-2}^{-1},z_{s+t-3}^{-1},\ldots,z_1^{-1}) \cdot f(z_1,\ldots,z_{s+t-2}), \\
\qt_{s,t}[f]&:=T_{s,t}(z_{1}^{-1};z_{s+t-1}^{-1},z_{s+t-2}^{-1},\ldots,z_2^{-1}) \cdot f(z_2,\ldots,z_{s+t-1}).
\end{align*}
Note that if we set $Q_{s,t}(z_1,\ldots,z_{s+t-1}):=P_{s,t}(z_{s+t-1}^{-1},\ldots,z_1^{-1})$, then 
$$
Q_{s,t}=\qs_{s,t}[Q_{s-1,t}]  \quad \text{ and } \quad Q_{s,t} = \qt_{s,t}[Q_{s,t-1}].
$$
From the definitions, one can deduce the following commutation properties; the proof is straightforward and left to the reader.

\begin{lemma}  
\label{commute}
Let $s,t$ be positive integers.
\begin{enumerate}
\item If $s, t \ge 1$, then $\ps_{s,t} \circ \pt_{s-1,t} = \pt_{s,t} \circ \ps_{s,t-1}$ and $\qs_{s,t} \circ \qt_{s-1,t} = \qt_{s,t} \circ \qs_{s,t-1}$.
\item If $t \ge 2$, then $\pt_{s,t} \circ \qt_{s,t-1} = \qt_{s,t} \circ \pt_{s,t-1}$.
\end{enumerate}
\end{lemma}

Moreover, we need the following identities, which follow from the fact that $S_{s,t}(z;z_1,\ldots,z_{s+t-2})$  and $T_{s,t}(z;z_1,\ldots,z_{s+t-2})$ are symmetric in $z_1,\ldots,z_{s+t-2}$ (the symbol $\widehat{z_i}$ indicates that $z_i$ is missing from the argument):
\begin{equation}
\label{step}
\begin{aligned}
\sym \ps_{s,t}[f]&= \sum_{i=1}^{s+t-1} S_{s,t}(z_{i};z_1,\ldots,\widehat{z_i},\ldots,z_{s+t-1}) \sym f(z_1,\ldots,\widehat{z_i},\ldots,z_{s+t-1}), \\
\sym \pt_{s,t}[f] &=\sum_{i=1}^{s+t-1} T_{s,t}(z_{i};z_1,\ldots,\widehat{z_i},\ldots,z_{s+t-1}) \sym f(z_1,\ldots,\widehat{z_i},\ldots,z_{s+t-1}), \\
\sym \qs_{s,t}[f]&= \sum_{i=1}^{s+t-1} S_{s,t}(z_{i}^{-1};z_1^{-1},\ldots,\widehat{z_i^{-1}},\ldots, z_{s+t-1}^{-1}) \sym f(z_1,\ldots,\widehat{z_i},\ldots, z_{s+t-1}), \\
\sym \qt_{s,t}[f]&=\sum_{i=1}^{s+t-1} T_{s,t}(z_{i}^{-1};z_1^{-1},\ldots,\widehat{z_i^{-1}},\ldots, z_{s+t-1}^{-1}) \sym f(z_1,\ldots,\widehat{z_i},\ldots, z_{s+t-1}).
\end{aligned}
\end{equation}

We consider words $w$ over the alphabet ${\mathcal A}:=\{\ps, \pt, \qs, \qt\}$ and let $|w|_S$ denote the number of occurrences of $\ps$ and $\qs$ in the word and $|w|_T$ denote the number of occurrences of $\pt$ and $\qt$. It is instructive to interpret these words as labelled lattice paths with starting point in the origin, step set $\{(1,0),(0,1)\}$ and labels $P, Q$. The letters $\ps$ and $\qs$ correspond to $(1,0)$-steps labelled with $P$ and $Q$, respectively, while the letters 
$\pt$ and $\qt$ correspond to $(0,1)$-steps. With this interpretation, $(|w|_S,|w|_T)$ is the endpoint of the path (see Figure \ref{fig:pq-path1}).

\begin{figure}[ht]
\begin{center}
\scalebox{1.5}{
\includegraphics{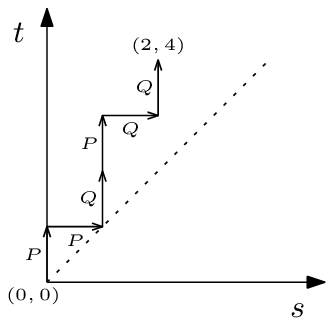}
}
\end{center}
\caption{Labelled lattice path corresponding to $w=(\pt,\ps,\qt,\pt,\qs,\qt)$.\label{fig:pq-path1}}
\end{figure} 
To every word $w$ of length $n$, we assign a rational function $F_{w}(z_1,\ldots,z_{n+1})$ as follows:  If $w$ is the empty word, then $F_{w}(z_1):=1$. Otherwise, if $L \in {\mathcal A}$ and $w$ is a word over ${\mathcal A}$, we set 
$$F_{w L}:= L_{|w L |_S+1,|w L|_T+1} [F_{w}].$$
For example, the rational function assigned to $w$ in Figure \ref{fig:pq-path1} is
\[
F_{w}(z_1,\ldots,z_7) = \qt_{3,5} \circ \qs_{3,4} \circ \pt_{2,4} \circ \qt_{2,3} \circ \ps_{2,2} \circ \pt_{1,2}[1]. 
\]
In this context, Lemma~\ref{commute} has the following meaning: on the one hand, we may swap two consecutive steps with the same label, and, one the other hand, we may swap two consecutive $(0,1)$-steps without changing the corresponding rational functions. For example, the rational functions corresponding to the words in Figure \ref{fig:pq-path1} and Figure \ref{fig:pq-path2} coincide. 
\begin{figure}[ht]
\begin{center}
\scalebox{1.5}{
\includegraphics{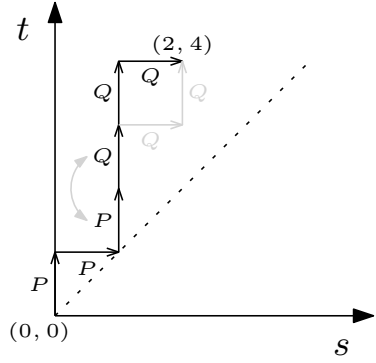}
}
\end{center}
\caption{Labelled lattice path corresponding to $\tilde{w}=(\pt,\ps,\pt,\qt,\qt,\qs)$.\label{fig:pq-path2}}
\end{figure}

\begin{proof}[Proof of Theorem~\ref{inductionstep}]
We assume 
\begin{equation}
\label{assumption}
R_{s,t}(z_1,\ldots,z_{s+t-1}) = R_{s,t}(z_1^{-1},\ldots,z_{s+t-1}^{-1})
\end{equation}
if $t=s$ and $t=s+1$. We show the following more general statement: Suppose 
$w_1, w_2$ are two words over ${\mathcal A}$ with $|w_1|_S=|w_2|_S$  and $ |w_1|_T=|w_2|_T$, and every prefix 
$w'_i$ of $w_i$ fulfills $|w'_i|_S \le  |w'_i|_T$, $i=1,2$.  (In the lattice paths language this means that $w_1$ and $w_2$ are both prefixes of
Dyck paths sharing the same endpoint; there is no restriction on the labels $P$ and $Q$.) Then
\begin{equation}
\label{id:general}
\sym F_{w_1} = \sym F_{w_2}.
\end{equation}
The assertion of the theorem then follows since $F_w = P_{|w|_S+1,|w|_T+1}$ if $w$ is a word over $\{\ps, \pt\}$ and   
$F_w = Q_{|w|_S+1,|w|_T+1}$ if $w$ is a word over $\{\qs,\qt\}$, and therefore
\begin{align*}
R_{s,t}(z_1,\ldots,z_{s+t-1}) &= \sym P_{s,t}(z_1,\ldots,z_{s+t-1}) = \sym Q_{s,t}(z_1,\ldots,z_{s+t-1})\\
 &= \sym P_{s,t}(z_{s+t-1}^{-1},\ldots,z_1^{-1})= R_{s,t}(z_1^{-1},\ldots,z_{s+t-1}^{-1}).
\end{align*}
The proof is by induction with respect to the length of the words; there is nothing to prove if the words are empty. Otherwise let 
$w_1, w_2$ be two words over ${\mathcal A}$ 
with $|w_1|_S=|w_2|_S=:s-1$  and $ |w_1|_T=|w_2|_T=:t-1$, and every prefix 
$w'_i$ of $w_i$ fulfills $|w'_i|_S \le  |w'_i|_T$, $i=1,2$. 
Note that the induction hypothesis and \eqref{step} imply that $\sym F_{w_i}$ only depends on the last letter of $w_i$ (and on $s$ and $t$ of course).
Thus the assertion follows if the last letters of 
$w_1$ and $w_2$ coincide; we assume that they differ in the following.

If $s=t$, then the assumption on the prefixes implies that the last letters 
of $w_1$ and $w_2$ are in $\{\ps,\qs\}$. W.l.o.g.\ we assume $w_1= w_1' \ps$ and $w_2 = w_2' \qs$. By the induction hypothesis and \eqref{step}, we have 
$\sym F_{w_1} = \sym P_{s,s}$ and $\sym F_{w_2} = \sym Q_{s,s}$. The assertion now follows from \eqref{assumption}, since $\sym P_{s,s}(z_1,\ldots,z_{2s-1})=R_{s,s}(z_1,\ldots,z_{2s-1})$ and 
$\sym Q_{s,s}(z_1,\ldots,z_{2s-1})=R_{s,s}(z_1^{-1},\ldots,z_{2s-1}^{-1})$.

If $s < t$, we show that we may assume that the last letters  of $w_1$ and $w_2$ are in $\{\pt, \qt\}$: if this is not true for the last letter $L_1$ of $w_i$, we may at least assume by the induction hypothesis and 
\eqref{step} that the penultimate letter $L_2$ is in $\{\pt,\qt\}$; to be more precise, we require $L_2=\pt$ if $L_1=\ps$ and $L_2=\qt$ if 
$L_1=\qs$; now, according to Lemma~\ref{commute}, we can interchange the last and the penultimate letter in this case. 

If $t=s+1$, then \eqref{id:general}
now follows from \eqref{assumption} in a similar fashion as in the case when $s=t$.

If $s+1 < t$, we may assume w.l.o.g.\ that the last letter of $w_1$ is $\pt$ and the last letter of $w_2$ is $\qt$. By the induction hypothesis 
and \eqref{step}, we may assume that the penultimate letter of $w_1$ is $\qt$. According to Lemma~\ref{commute}, we can interchange the last
and the penultimate letter of $w_1$ and the assertion follows also in this case.
\end{proof}

\section{Some remarks on the case \texorpdfstring{$s=0$}{s = 0} in Conjecture \texorpdfstring{\ref{conj}}{1.1}}

If $s=0$ in Conjecture~\ref{conj}, then the rational function simplifies to 
\begin{equation}
\label{original}
\prod_{1 \le i < j \le n} \frac{z_i^{-1} + z_j -1}{1-z_i z_j^{-1}}
\end{equation}
where $n=t-1$. This raises the question of whether there are also other 
rational functions $T(x,y)$ such that symmetrizing $\prod\limits_{1 \le i < j \le n} T(z_i,z_j)$ leads to a Laurent polynomial that
is invariant under replacing $z_i$ by $z_i^{-1}$. Computer experiments suggest that this is the case for 
\begin{equation}
\label{R}
T(x,y) = \frac{[a(x^{-1}+y)+c][b(x+y^{-1})+c]}{1-x y^{-1}} + a b x^{-1} y +d
\end{equation}
where $a,b,c,d \in \mathbb{C}$. (Since $T(x,y)=T(y^{-1},x^{-1})$ it is obvious that the symmetrized function is invariant under replacing all $z_i$ simultaneously by 
$z_i^{-1}$.) 

In case $a=0$ it can be shown with a degree argument that symmetrizing leads to a function that does not depend on $z_1,z_2,\ldots,z_n$. (In fact, this is also true for $\prod\limits_{1 \le i < j \le n} \frac{A z_i z_j 
+ B z_i + C z_j + D}{z_j - z_i}$, and our case is obtained by specializing $A=b c,B=-d,C=c^2+d,D=bc$.)

In case $T(x,y)=\frac{x^{-1}+y}{1-x y^{-1}}$ (which is obtained from the above function by setting $b=d=0$ then dividing by $c$ and setting $a=1, c=0$ afterwards) this is also easy to see, since the symmetrized function can be computed explicitly as follows:
\begin{align*}
\sym & \prod_{1 \le i < j \le n} \frac{z_i^{-1} + z_j}{1-z_i z_j^{-1}} 
= \sym \prod_{1 \le i < j \le n} \frac{z_i^{-1}  z_j (1+z_i z_j)}{z_j - z_i} \\
&= \prod_{1 \le i < j \le n} (1 + z_i z_j) \prod_{i=1}^{n} z_i^{-n+1} 
\sym  \frac{\prod_{i=1}^{n} z_i^{2i-2}}{\prod_{1 \le i < j \le n} (z_j - z_i)} \\
&= \prod_{1 \le i < j \le n} (1 + z_i z_j) \prod_{i=1}^{n} z_i^{-n+1} \frac{\det_{1 \le i, j \le n} ((z_i^{2})^{j-1})}{\prod_{1 \le i < j \le n} (z_j - z_i)} \\
&= \prod_{1 \le i < j \le n} (1 + z_i z_j) \prod_{i=1}^{n} z_i^{-n+1} \prod_{1 \le i < j \le n} \frac{z_j^{2} - z_i^{2}}{z_j-z_i} =  \prod_{1 \le i < j \le n} (1 + z_i z_j) (z_i + z_j) \prod_{i=1}^{n} z_i^{-n+1}.
\end{align*}

We come back to \eqref{original}. In our computer experiments we observed that if we specialize $z_1=z_2=\ldots=z_n=1$ in the symmetrized function then 
we obtain the number of $(2n+1) \times (2n+1)$ Vertically Symmetric Alternating Sign Matrices. Next we aim to prove a generalization of this.

For this purpose we consider the following slight generalization of 
$\alpha(n;k_1,\ldots,k_n)$ for non-negative integers $m$:
$$
\alpha_m(n;k_1,\ldots,k_n)= \prod_{1 \le p < q \le n} (\id + E_{k_p} E_{k_q} + (X-2) E_{k_p}) \det_{1 \le i, j \le n} \left( \binom{k_i}{j-1 + m \, \delta_{j,n} } \right)
$$
In \cite{FischerNumberOfMT} it was shown that $\alpha_0(n;k_1,\ldots,k_n)=\alpha(n;k_1,\ldots,k_n)$
if $X=1$. For $k_1 \le t \le k_n$, choose $c_m \in \mathbb{C}$, almost all of them zero, 
such that the polynomial $\sum_{m=0}^{\infty} c_m \binom{x}{m}$ is $1$ if $x=t$ and $0$ if $x \in \{k_1,k_1+1,\ldots,k_n\} \setminus \{t\}$. In \cite{FischerSimplifiedProof} it was shown that 
\begin{equation}
\label{top}
\sum_{m=0}^{\infty} c_m \alpha_{m}(n;k_1,k_2,\ldots,k_n)
\end{equation}
is the generating function ($X$ is the variable) of Monotone Triangles $(a_{i,j})_{1 \le j \le i \le n}$ 
with bottom row $(k_1,k_2,\ldots,k_n)$ and  top entry $t$ with respect to the occurrences of the ``local pattern'' $a_{i+1,j} < a_{i,j} < a_{i+1,j+1}$. 
In fact, these patterns correspond to the $-1$s in the corresponding Alternating Sign Matrix if $(k_1,\ldots,k_n)=(1,2,\ldots,n)$.

\begin{proposition}
\label{eval1}
Fix integers $k_1,k_2,\ldots,k_n$ and a non-negative integer $m$, and define
\[
Q(z_1,\ldots,z_n) := \sym \left( \prod\limits_{i=1}^{n} z_i^{k_i} \prod\limits_{1 \le i < j \le n} \frac{1 + z_i z_j + (X-2) z_i}{z_j-z_i} \right).
\]
Then 
$$
\alpha_m(n;k_1,\ldots,k_n) = \left. Q(1,1,\ldots,1,E_{l}) \binom{l}{m}\right|_{l=0}.
$$
\end{proposition}

\begin{proof}
We set $P(z_1,\ldots,z_n)= \prod\limits_{i=1}^{n} z_i^{k_i} \prod\limits_{1 \le i < j \le n} ( 1 + z_i z_j + (X-2) z_i)$. Then 
\begin{align*}
\alpha_m(n;k_1,\ldots,k_n)&= \left. E_{l_1}^{k_1} \cdots E_{l_n}^{k_n} \alpha_{m}(n;l_1,\ldots,l_n) \right|_{l_1=\ldots=l_n=0} \\ &=
\left. P(E_{l_1},\ldots,E_{l_n}) \sum_{\sigma \in {\mathcal S}_n} \sgn \sigma \binom{l_{\sigma(1)}}{0} \ldots \binom{l_{\sigma(n-1)}}{n-2} 
\binom{l_{\sigma(n)}}{n-1+m} \right|_{l_1=\ldots=l_n=0}.
\end{align*}
With $P(z_1,\ldots,z_n) = \sum_{i_1,i_2,\ldots,i_n} p_{i_1,i_2,\ldots,i_n} z_1^{i_1} z_2^{i_2} \cdots z_n^{i_n}$, this is equal to 
\begin{multline*}
\sum_{\sigma \in \mathcal{S}_n, i_1,i_2,\ldots,i_n} \sgn \sigma \, p_{i_1,i_2,\ldots,i_n} \binom{i_{\sigma(1)}}{0} 
\ldots \binom{i_{\sigma(n-1)}}{n-2} 
\binom{i_{\sigma(n)}}{n-1+m} \\
= \left. \sum_{\sigma \in \mathcal{S}_n, i_1,i_2,\ldots,i_n} \sgn \sigma \,
p_{i_1,i_2,\ldots,i_n} E_{l_1}^{i_{\sigma(1)}} \ldots E_{l_n}^{i_{\sigma(n)}}
\binom{l_1}{0} \ldots \binom{l_{n-1}}{n-2} \binom{l_n}{n-1+m} \right|_{l_1=\ldots=l_n=0}
\\
= \left. \sum_{\sigma \in \mathcal{S}_n, i_1,i_2,\ldots,i_n} \sgn \sigma \,
p_{i_1,i_2,\ldots,i_n} E_{l_{\sigma^{-1}(1)}}^{i_1} \ldots E_{l_{\sigma^{-1}(n)}}^{i_n}
\binom{l_1}{0} \ldots \binom{l_{n-1}}{n-2} \binom{l_n}{n-1+m} \right|_{l_1=\ldots=l_n=0}
\\
= \left. \asym P(E_{l_1},\ldots,E_{l_n}) \binom{l_1}{0} \ldots \binom{l_{n-1}}{n-2} \binom{l_n}{n-1+m} \right|_{l_1=\ldots=l_n=0}.
\end{multline*}
By definition, 
$$
\asym P(z_1,\ldots,z_n) = Q(z_1,\ldots,z_n) \prod_{1 \le i < j \le n} (z_j-z_i).
$$
Now we can conclude that
\begin{align*}
&\alpha_m(k_1,\ldots,k_n) = \left. Q(E_{l_1},E_{l_2},\ldots,E_{l_n}) \prod_{1 \le i < j \le n} (E_{l_j} - E_{l_i}) \binom{l_1}{0} \ldots \binom{l_{n-1}}{n-2}  \binom{l_n}{n-1+m} \right|_{l_1=\ldots=l_n=0} \\
&\quad=
 \left. Q(E_{l_1},E_{l_2},\ldots,E_{l_n}) \prod_{1 \le i < j \le n} (\Delta_{l_j} - \Delta_{l_i}) \binom{l_1}{0} \ldots \binom{l_{n-1}}{n-2} \ldots \binom{l_n}{n-1+m} \right|_{l_1=\ldots=l_n=0} \\
 &\quad=
 \left. Q(E_{l_1},E_{l_2},\ldots,E_{l_n}) \det_{1 \leq i,j \leq n} \left( \Delta_{l_i}^{j-1} \right) \binom{l_1}{0} \ldots \binom{l_{n-1}}{n-2} \ldots \binom{l_n}{n-1+m} \right|_{l_1=\ldots=l_n=0} \\
&\quad=
 \left. Q(E_{l_1},E_{l_2},\ldots,E_{l_n}) \sum_{\sigma \in \mathcal{S}_n} \sgn \sigma \Delta_{l_1}^{\sigma(1)-1} \ldots \Delta_{l_n}^{\sigma(n)-1} \binom{l_1}{0} \ldots  \binom{l_{n-1}}{n-2} \binom{l_n}{n-1+m} \right|_{l_1=\ldots=l_n=0} \\
&\quad=\left. Q(1,1,\ldots,1,E_{l_n}) \binom{l_n}{m} \right|_{l_n=0}, 
\end{align*} 
since $\Delta_{l_1}^{\sigma(1)-1} \ldots \Delta_{l_n}^{\sigma(n)-1} \binom{l_1}{0} \ldots  \binom{l_{n-1}}{n-2} \binom{l_n}{n-1+m}=0$ except when $\sigma=\id$.
\end{proof}

\begin{corollary} 
\label{cor_specialised}
Let $k_1,k_2,\ldots,k_n,m$ and $Q(z_1,\ldots,z_n)$ 
be as in Proposition~\ref{eval1}. Then the coefficient of $z^t X^k$ in 
$Q(1,1,\ldots,1,z)$ is the number of Monotone Triangles $(a_{i,j})_{1 \le j \le i \le n}$ with bottom row $k_1,k_2,\ldots,k_n$, top entry $t$ and $k$  occurrences of the local pattern   $a_{i+1,j} < a_{i,j} < a_{i+1,j+1}$. 
\end{corollary}

\begin{proof} We fix $t$ and observe that the combination of \eqref{top} and Proposition~\ref{eval1} implies that 
$$
\sum_{m \ge 0} c_m \left. Q(1,1,\ldots,1,E_l) \binom{l}{m} \right|_{l=0}
$$ 
is the generating function described after \eqref{top}.
Now, if we suppose 
$$
Q(1,1,\ldots,1,z) = \sum_{s,k} b_{s,k} z^s X^k, 
$$
then we see that this generating function is equal to 
\begin{multline*}
\sum_{m \ge 0} c_m  \left. \sum_{s,k} b_{s,k} E_l^s X^k \binom{l}{m} \right|_{l=0} 
= \sum_{m \ge 0} c_m  \sum_{s,k} b_{s,k} X^k \binom{s}{m} 
= \sum_{s,k} b_{s,k} X^k \sum_{m \ge 0} c_m \binom{s}{m}  \\
= \sum_{s,k} b_{s,k} X^k \delta_{s,t}   = \sum_{k} b_{t,k} X^k,
\end{multline*}
where the third equality follows from the choice of the coefficients $c_m$.
\end{proof}

A short calculation shows that $\sym$ applied to \eqref{original} is equal to $\prod\limits_{i=1}^{n} z_i^{-n+1} Q(z_1,\ldots,z_n)$ if we set $k_i=2(i-1)$ and $X=1$ in $Q(z_1,\ldots,z_n)$. If we also specialize $z_1=\dots=z_{n-1}=1$, then Conjecture~\ref{conj} implies $Q(1,\ldots,1,z)= z^{2n-2} Q(1,\ldots,1,z^{-1})$. However, by Corollary~\ref{cor_specialised}, this is just the trivial fact that the number of Monotone Triangles with bottom row 
$(0,2,4,\ldots,2n-2)$ and top entry $t$ is equal to the number of Monotone Triangles with bottom row $(0,2,4,\ldots,2n-2)$ and top entry $2n-2-t$, or, equivalently, that the number of $(2n+1) \times (2n+1)$ Vertically Symmetric Alternating Sign Matrices with a $1$ in position $(t,1)$ equals the number of 
$(2n+1) \times (2n+1)$ Vertically Symmetric Alternating Sign Matrices with a $1$ in position $(2n+1-t,1)$. So in the special case $s=0$, Conjecture~\ref{conj} is a generalization of this obvious symmetry.  

Finally, we want to remark that the symmetrized functions under consideration in Proposition~\ref{eval1} can easily be computed recursively. For instance, considering the case of Vertically Symmetric Alternating Sign Matrices, let 
$$
\vs(X;z_1,\ldots,z_n) = \sym \prod_{i=1}^{n} z_i^{2i-2} 
\prod_{1 \le i < j \le n} \frac{1+z_i (X-2) + z_i z_j}{z_j-z_i}.
$$
Then 
$$
\vs(X;z_1,\ldots,z_n)=\sum_{j=1}^{n} z_j^{2n-2} \prod_{1 \le i \le n, i \not= j}
\frac{1+ z_i (X-2) + z_i z_j}{z_j - z_i} \vs(X;z_1,\ldots,\widehat{z_j},\ldots,z_n).
$$
Similarly, in the case of ordinary Alternating Sign Matrices, let 
$$
\asm(X;z_1,\ldots,z_n) = \sym \prod_{i=1}^{n} z_i^{i-1} 
\prod_{1 \le i < j \le n} \frac{1+z_i (X-2) + z_i z_j}{z_j-z_i}.
$$
Then 
$$
\asm(X;z_1,\ldots,z_n)=\sum_{j=1}^{n} z_j^{n-1} \prod_{1 \le i \le n, i \not= j}
\frac{1+ z_i (X-2) + z_i z_j}{z_j - z_i} \asm(X;z_1,\ldots,\widehat{z_j},\ldots,z_n).
$$
Let us conclude by mentioning that in order to reprove the formula for the number of Vertically Symmetric Alternating Sign Matrices of given size,
it would suffice to compute 
$\vs(1;1,1,\ldots,1)$, while the ordinary Alternating Sign Matrix Theorem is equivalent to computing 
$\asm(1;1,1,\ldots,1)$.

\bibliographystyle{alpha}
\bibliography{bib130314}
\end{document}